\documentclass[a4paper,11pt]{amsart}
\usepackage{mathrsfs}
\usepackage[all]{xy}
\usepackage{amsmath,amssymb,amscd,bbm,amsthm,mathrsfs}
\usepackage{graphicx}
\newtheorem{thm}{Theorem}[section]
\newtheorem{lem}{Lemma}[section]
\newtheorem{cor}{Corollary}[section]
\newtheorem{prop}{Proposition}[section]
\newtheorem{rem}{Remark}[section]

\begin{document}
\numberwithin{equation}{section}

\title[The tangential $k$-Cauchy-Fueter complexes and Hartogs' phenomenon  ]{The tangential $k$-Cauchy-Fueter complexes and Hartogs' phenomenon over the right quaternionic Heisenberg group}

\author{ Yun Shi and  Wei Wang}

\begin{abstract}
We construct the tangential $k$-Cauchy-Fueter complexes   on the right quaternionic Heisenberg group, as the   quaternionic counterpart of $\overline{\partial}_b$-complex on the  Heisenberg group in the theory of several complex variables.   We can use the $L^2$ estimate to solve the nonhomogeneous tangential $k$-Cauchy-Fueter equation under the compatibility condition over  this  group modulo a lattice. This solution has an important vanishing property when the group is   higher  dimensional. It allows us to prove the  Hartogs' extension phenomenon for $k$-CF functions,   which are the   quaternionic counterpart of CR functions.
\end{abstract}
\thanks{The first author is partially supported by National Nature Science
Foundation
  in China (No. 11801508; No. 11571305); The second author is partially supported by National Nature Science
Foundation
  in China (No. 11571305)\\
Department of Mathematics, Zhejiang University of Science and Technology, Hangzhou 310023, China,
E-mail: hzxjhs1987@163.com;
 Department of Mathematics, Zhejiang University,  Hangzhou 310027, China,
E-mail: wwang@zju.edu.cn}
\maketitle

\section{Introduction}
The $\overline{\partial}$-complex  plays an important role in the theory of several complex variables since many important results for holomorphic functions can be obtained by solving nonhomogeneous $\overline{\partial}$-equation. We obtain $\overline{\partial}_b$-complex when it is restricted to a CR submanifold,  and many important results for CR functions can be also  obtained by solving $\overline{\partial}_b$-equation.
In general, for a differential complex, there is an abstract way to obtain a boundary  complex restricted to a submanifold, which is written down in terms of quotient sheafs (cf. e.g. \cite{Andreotti,Andreotti2,Nacinovich}).

In quaternionic analysis we now know  the $k$-Cauchy-Fueter complex explicitly (cf. \cite{adams2,Ba,bS,bures,CSS,SSSV,wang15,wang24} and references therein), which are used to show several interesting properties of $k$-regular functions (cf. \cite{Colombo,wang15,wang23}  and references therein). When restricted to a quadratic hypersurface in $\mathbb{H}^{n+1},$ we have the tangential $k$-Cauchy-Fueter operators and $k$-CF functions (cf. \cite{wang2} for $k=1,n=2$),  corresponding to $\overline{\partial}_b$ and CR functions over a CR manifold.  In this paper, we will consider their restriction to a model quadratic hypersurface
\begin{equation}\label{S}
\mathcal{S}:=\{(q',q_{n+1})\in\mathbb{H}^{n}\times\mathbb{H}:\rho(q',q_{n+1})=0\}
\end{equation}
 in $\mathbb{H}^{n+1},$ where
 \begin{equation}\label{hyper}
  \rho(q',q_{n+1}):={\rm Re}\,q_{n+1}-\phi(q'),\qquad  \phi(q'):=\sum_{l=0}^{n-1}\left(-3x^2_{4l+1}+x^2_{4l+2}
+x^2_{4l+3}+x^2_{4l+4}\right).
 \end{equation}
 Here we write $q'=(\cdots,q_l,\cdots),q_l=x_{4l+1}+\mathbf{i}x_{4l+2}+
\mathbf{j}x_{4l+3}+\mathbf{k}x_{4l+4}.$
This hypersurface has the structure of  the \emph{right quaternionic Heisenberg group}  $\mathscr{H}=\mathbb{H}^n\times\rm{Im}\ \mathbb{H}$ with  the multiplication given by
\begin{align}\label{hei}
(x,t)\cdot(y,s)=\left(x+y,t+s+2{\rm{Im}} ({x}\overline{y})\right),
\end{align}
where $x,y\in \mathbb{H}^{n}$ and ${t},s\in \rm{Im}\ \mathbb{H}.$ We construct a family of  differential complexes on $ \mathscr{H} $,  the \emph{tangential $k$-Cauchy-Fueter complexes}, given by
\begin{equation}\begin{aligned}\label{cf}
0\rightarrow C^{\infty}(\Omega,\mathscr{V}_0)
&\xrightarrow{\mathscr{D}_{0}} C^{\infty}(\Omega,\mathscr{V}_1)\xrightarrow{\mathscr{D}_{1}} C^{\infty}(\Omega,\mathscr{V}_2)\rightarrow \cdots\xrightarrow{\mathscr{D}_{2n-2}} C^{\infty}(\Omega,\mathscr{V}_{2n-1})\rightarrow0,
\end{aligned}\end{equation}
for  a domain  $\Omega $   in  $\mathscr{H} $, where
\begin{equation}\begin{aligned}\label{V}
\mathscr{V}_j:=&\odot^{k-j}\mathbb{C}^{2}\otimes
\wedge^j\mathbb{C}^{2n},\qquad \qquad \quad j=0,1,\cdots, k,
\\
\mathscr{V}_{j}:=&\odot^{j-k-1}\mathbb{C}^{2}\otimes\wedge^{j+1}\mathbb{C}^{2n},
\qquad\quad j=k+1,\cdots,2n-1,
\end{aligned}\end{equation}
for fixed $k=0,1,\cdots$, and  $\odot^{p}\mathbb{C}^{2}$ is the $p$-th symmetric power of $\mathbb{C}^2.$ They are  the quaternionic counterpart  of   $\bar{\partial}_b$-complex  over the Heisenberg group in  the theory of  several complex variables. They have the same form as the $k$-Cauchy-Fueter complexes on $\mathbb{H}^n$ (cf. Remark \ref{rem:kk}), but $\mathscr{D}_{j}$'s are given in terms of left invariant vector fields    (\ref{eq:D-j-1}) (\ref{eq:D-j-2}) (\ref{eq:D-j-3}), which are differential operators of variable coefficients. So we can not use the computational algebraic method in \cite{Colombo} to construct these complexes. This family of  complexes are natural in the sense that they can be viewed as restriction to the hypersurface $\mathcal{S}$ of the complexes on $\mathbb{H}^{n+1},$ but  not natural in the sense that they are not invariant under the conformal transformation group ${\rm Sp}(2(n+2),1)$ of $\mathscr H$ (cf. Subsection \ref{sub25}).

$\mathscr{D}_{0}$ in  (\ref{cf}) is called the \emph{tangential $k$-Cauchy-Fueter operator}. A $\odot^{k}\mathbb{C}^{2}$-valued distribution $f$ on $\Omega$ is called \emph{$k$-CF}  if $\mathscr{D}_{0}f=0$ in the sense of distributions. The space of all $k$-CF functions on $\Omega$ is denoted by $\mathcal{A}_k(\Omega).$ A $1$-CF function is also called \emph{anti-CRF function} in \cite{Ivanov,Ivanov10}. Such functions play an important role in the  study of  pseudo-Einstein equation over the quaternionic Heisenberg group  \cite{Ivanov10}.

On the other hand, when the hypersurface is the boundary of the Siegel upper half space, i.e. the defining function in (\ref{S}) is given by
$$\rho={\rm Re}\, q_{n+1}-|q'|^2,$$ the corresponding group is  the \emph{left quaternionic Heisenberg group}
$\widetilde{\mathscr H}:=\mathbb{H}^n\times{\rm Im}\,\mathbb{H}$ with the multiplication given by
\begin{align}\label{hei'}
(x,t)\cdot(y,s)=\left(x+y,t+s+2{\rm{Im}} (\overline{x}{y})\right).
\end{align}
We already know the tangential $k$-Cauchy-Fueter complex (cf. \cite[Theorem 1.0.1]{wang1}) on the left quaternionic Heisenberg group  by using the twistor method  (see also \cite{BE,PS} for constructing complexes by this method) . But in this case $\wedge^j\mathbb{C}^{2n}$ in (\ref{V}) must be replaced by   the irreducible representation  of $\mathfrak{sp}(2n,\mathbb{C})$ with the highest weight to be the $j$-th fundamental weight (cf. Subsection \ref{sub25}).
It is  more complicated than   the  right quaternionic case. So we only consider  the right quaternionic Heisenberg group in this paper.
We see that when restricted to different submanifolds, we get different differential complexes. This is a new phenomenon compared to  several complex variables, where expressions of  $\overline{\partial}_b$-complex for different CR submanifolds are the same.   It is an interesting problem to write down explicitly the tangential $k$-Cauchy-Fueter complexes for all quadratic hypersurfaces in $\mathbb{H}^{n+1}$ (cf. \cite{wang2} for such hypersurfaces).

In this paper we prove Hartogs' phenomenon for $k$-CF functions over right quaternionic Heisenberg group.

\begin{thm}\label{hartogs}
Let $\Omega$ be a bounded open set in the right quaternionic Heisenberg group $\mathscr{H}$ with  ${\rm dim}\ \mathscr{H}\geq19,$ and let $K$ be a compact subset of $\Omega$ such that $\Omega\setminus K$ is connected. Then for each $u\in\mathcal{A}_k(\Omega\setminus K),$  $k=2,3,\cdots,$ we can find $U\in\mathcal{A}_k(\Omega)$ such that $U=u$ in $\Omega\setminus K.$
\end{thm}
The restriction    of ${\rm dim}\ \mathscr{H}  $ and $k $ in this theorem comes  from the technical  difficulty to establish the  $L^2$ estimate in the remaining cases.
A form of Hartogs' phenomenon was proved for many elliptic differential systems  (cf. \cite{Colombo,Palamodov} and references therein). Notably  in our case  $\mathscr{D}_0$ as a matrix-valued horizontal vector field  is not an elliptic  system, and (\ref{cf}) is not an elliptic complex. This is because   symbols of $\mathscr{D}_j$'s  vanish  at the cotangent vectors   annhilating horizontal vector fields.

In  several complex variables we have  deep Hartogs-Bochner effect for CR functions on  CR submanifolds,  which   are usually proved by using integral representation formulae (cf. \cite{Henkin,Laurent,Porten} and references therein for further development of  this effect).   But in the quaternionic case, the integral representation formulae are not sufficiently developed, and only Bochner-Martinelli type formulae are known (cf. \cite{WangH,wang15}). As in the theory of several complex variables, the  formulae with Bochner-Martinelli type kernels  are not good enough to prove the extension phenomenon.

Given a differential complex, it is a fundamental problem   to investigate  its cohomology group or its Poincar\'e lemma over a domain (cf. e.g. \cite{Brinkschulte,Hill}). In particular, we hope to    solve the nonhomogeneous tangential $k$-Cauchy-Fueter equation
\begin{align}\label{equ}
\mathscr{D}_0u=f,
\end{align}
for    $f\in L^2(\mathscr{H},\mathscr{V}_1)$,  under  the compatibility condition
\begin{align}\label{equ1}
\mathscr{D}_1f=0,
\end{align}
i.e. $f$ is $\mathscr{D}_1$-closed. If we can find compactly supported solution of (\ref{equ})-(\ref{equ1}) when $f$ is compactly supported, it is a standard  procedure  to derive  Hartogs' phenomenon (cf. e.g. \cite{Homander,wang15}). One way to solve (\ref{equ})-(\ref{equ1}) is to consider the associated Hodge-Laplacian
\begin{equation}\begin{aligned}\label{Box}
\Box_1 =\mathscr{D}_0\mathscr{D}_0^*+\mathscr{D}_1^*\mathscr{D}_1:L^2(\mathscr{H},\mathscr{V}_1)\rightarrow L^2(\mathscr{H}, \mathscr{V}_1).
\end{aligned}
\end{equation}
By identifying $\mathscr{V}_1=\odot^{k-1}\mathbb{C}^2\otimes\mathbb{C}^{2n}$  with $\mathbb{C}^{2nk},$ we can see that $\Box_1$ is a $(2kn)\times(2kn)$-matrix valued differential operator of second order, which is not diagonal (cf. Appendix   for the expression in  the case $n=2,k=2$). So it is not easy to verify the subellipticity of $\Box_1$ and find its fundamental solution. While in the complex case, the Hodge-Laplacian associated to $\overline{\partial}_b$-complex is diagonal and it is  easy to find its fundamental solution (cf. \cite{Folland}).

By using the $L^2$ method, we  establish the  following  estimate: when ${\rm dim}\ \mathscr{H}\geq19,$ there exists some constant $c>0$ such that
\begin{align}\label{1.9'}
\|\mathscr{D}_{0}^*f\|^2+\|\mathscr{D}_{1}f\|^2\geq c\langle\Delta_bf,f\rangle
\end{align}
for  $f\in C^2\left(\mathscr{H},\mathscr{V}_1\right)\cap L^2\left(\mathscr{H},\mathscr{V}_1\right),$   where $\Delta_b$ is the SubLaplacian on the right quaternionic Heisenberg group. But $\langle\Delta_bf,f\rangle$ does not control the $L^2$ norm of $f.$ It only controls $\|f\|^2_{L^{\frac{Q+2}{Q-2}}}$  by the well known Sobolev inequality  \cite{Ivanov10}, where $Q=4n+6$ is the homogeneous dimension of $\mathscr{H}.$ To avoid this difficulty, we consider the locally flat compact manifold $\mathscr{H}/\mathscr{H}_{\mathbb{Z}},$ where
\begin{align}\label{dis h}
\mathscr{H}_{\mathbb{Z}}:=\mathbb{Z}^{4n+3}
\end{align}
is a lattice of $\mathscr{H}.$ It is a spherical qc manifold (cf. \cite{Shi}).
Because the selfadjoint subelliptic operator $\Delta_b$ over a compact manifold has discrete spectra, $\langle\Delta_bf,f\rangle$ controls the $L^2$ norm of $f$ for  $f\perp {\rm ker}\,\Delta_b.$ Moreover, by the    Poincar\'e-type  inequality we can show  ${\rm ker}\,\Delta_b$ consisting  of constant vectors. Namely there exists some $c''>0$ such that
\begin{align}\label{1.9}
\langle\Delta_bf,f\rangle\geq c''\|f\|^2
\end{align}
for $f\in C^2(\mathscr{H}/\mathscr{H}_{\mathbb{Z}},\mathscr{V}_1)$ and  $f\perp$ {constant} vectors. It is a standard way to use the $L^2$ estimate to solve the non-homogeneous tangential $k$-Cauchy-Fueter equation (\ref{equ})-(\ref{equ1}) on $\mathscr{H}/\mathscr{H}_\mathbb{Z}$. The solution has an important  vanishing property  which allows us to prove Hartogs' phenomenon. See also \cite{Folland} for the existence theorem for $\overline{\partial}_b$-equation over compact CR manifolds by establishing a priori estimate.

In Section 2, we give preliminaries on the right  quaternionic Heisenberg group, the horizontal complex vector fields $Z_A^{A'}$'s and nice behavior of their commutators. We also give the definition of the tangential $k$-Cauchy-Fueter operators and their basic properties. It is checked directly that (\ref{cf})-(\ref{V}) is a complex. We compare the complexes on the left and right  quaternionic Heisenberg groups.  In Section 3, we use integration by part and  Poincar\'e-type inequality to show the $L^2$ estimate (\ref{1.9'}) (\ref{1.9}) for the tangential $k$-Cauchy-Fueter operator. In Section 4, we use the $L^2$ estimate to solve the nonhomogeneous tangential $k$-Cauchy-Fueter equation (\ref{equ})-(\ref{equ1}) over the quotient manifold $\mathscr{H}/\mathscr{H}_{\mathbb{Z}} $, and derive the Hartogs' phenomenon. In Section 5, we construct the nilpotent Lie groups of step two associated to   quadratic hypersurfaces. By constructing a diffeomorphism from the group $\mathscr{H}$ to  the hypersurface $\mathcal{S}$ in (\ref{S}), we show that the pushforward of the tangential $k$-Cauchy-Fueter operator  on the group $\mathscr{H} $  coincides with the restriction of the $k$-Cauchy-Fueter operator  on $\mathbb{H}^{n+1}$ to this hypersurface.  Therefore the  restriction of a $k$-regular functions to  $\mathcal{S}$ is  $k$-CF   on $\mathscr{H}.$ $k$-CF functions are abundant because so are $k$-regular functions on $\mathbb{H}^{n+1}$ \cite{Kang}. In the Appendix, we  give the expression of $\Box_1$ for $n=2,k=2.$

\section{The  tangential $k$-Cauchy-Fueter complexes}
\subsection{The right quaternionic Heisenberg group $\mathscr{H}$ and the locally flat compact manifold $\mathscr{H}/\mathscr{H}_\mathbb{Z}$}
The multiplication of the right quaternionic Heisenberg group $\mathscr{H}$ can be written in terms of  real variables (cf.  \cite[(2.13)]{wang22}) as
 \begin{align}\label{mul}
(x,t)\cdot({y},{s})= {\left(x+y,t_{\beta}+s_{\beta}+2
\sum_{l=0}^{n-1}\sum_{j,k=1}^{4}B_{kj}^{\beta}x_{4l+k}y_{4l+j}\right)},
\end{align}
for $x,y\in \mathbb{R}^{4n},\ t,s\in \mathbb{R}^{3},\ \beta=1,2,3,$ where $B_{kj}^{\beta}$ is the $(k,j)$-th entry of the following matrices
\begin{equation}\label{2.14}
\begin{aligned}
B^{1}:=\left(\begin{array}{cccc} 0 & -1 & 0 &0\\ 1& 0& 0& 0\\ 0& 0&0& -1\\0 &0& 1 &0\end{array}\right),   B^{2}:=\left(\begin{array}{cccc} 0 & 0 &
-1 &0\\ 0& 0& 0& 1\\ 1& 0&0& 0\\0 &-1& 0 &0\end{array}\right),
B^{3}:=\left(\begin{array}{cccc} 0 & 0 & 0 &-1\\ 0& 0& -1& 0\\ 0& 1&0& 0\\1 &0& 0 &0\end{array}\right),
\end{aligned}
\end{equation}
satisfying the commutating relation of quaternions
$(B^{1})^{2}=(B^{2})^{2}=(B^{3})^{2}=-id,\ B^{1}B^{2}=B^{3}.$
This is because for $x=x_{1}+x_{2}\textbf{i}+x_{3}\textbf{j}+x_{4}\textbf{k}$ and $x'=x'_{1}+x'_{2}\textbf{i}+x'_{3}\textbf{j}+x'_{4}\textbf{k},$ we have
\begin{equation*}\label{584}
\begin{aligned}
{\rm Im}(x\overline{x'})
&=(-x_{1}x'_{2}+x_{2}x'_{1}-x_{3}x'_{4}+x_{4}x'_{3})\textbf{i}
+(-x_{1}x'_{3}+x_{3}x'_{1}+x_{2}x'_{4}-x_{4}x'_{2})\textbf{j}\\
&\quad+(-x_{1}x'_{4}+x_{4}x'_{1}-x_{2}x'_{3}+x_{3}x'_{2})\textbf{k}
=\sum_{\beta=1}^{3}\sum_{k,j=1}^{4}B_{kj}^{\beta}x_{k}x'_{j}\textbf{i}_{\beta},
\end{aligned}
\end{equation*}where  $\textbf{i}_0=1,\textbf{i}_1=\textbf{i},
\textbf{i}_2=\textbf{j},\textbf{i}_3=\textbf{k}$.
For fixed point $(y,s)\in\mathscr{H},$ the \emph{left translate}
$\tau_{(y,s)}:\mathscr{H}\longrightarrow\mathscr{H},$ $
(x,t)\longmapsto (y,s)\cdot(x,t),$
is an affine  transformation given by a lower triangular matrix by (\ref{mul}). So the Lebesgue  measure on $\mathbb{R}^{4n+3}$ is an invariant measure under the left translation of  $\mathscr{H}.$ Recall that we have the following left invariant vector fields on $\mathscr{H} $:
\begin{equation} \label{eq:left-invariant}
(Y_{a}f)(y ,s )=\left.\frac{\hbox{d}}{\hbox{d}t}f((y ,s )(te_{a},0))
\right|_{t=0} , \qquad a=1,2,\ldots,4n,
 \end{equation}
where $e_{a}$ is $(0,\cdots,1,\cdots,0)$ with only the $a$-th entry equal to $1$.
Then
\begin{align}\label{2.43}
Y_{4l+j}:=\frac{\partial}{\partial y_{4l+j}}+2\sum_{\beta=1}^{3}\sum_{k=1}^{4}B^{\beta}_{kj}y_{4l+k}
\frac{\partial}{\partial s_{\beta}},
\end{align}
  whose brackets are
\begin{equation} \label{2.16'}
  [Y_{4l+k},Y_{4l+j}]  =4
\sum_{\beta=1}^{3}B_{kj}^{\beta}\partial_{s_{\beta}},
\qquad {\rm and }\qquad
[Y_{4l+k},Y_{4l'+j}] =0\quad{\rm for}\  l\neq l',
 \end{equation}
where  $l,l'=0,1,\cdots,n-1,$ $j,k=1,\cdots,4.$ The \emph{SubLaplacian} is defined as
\begin{align}\label{sub}
\Delta_b:=-\sum_{a=1}^{4n}Y_{a}^2.
\end{align}

The norm of the right quaternionic Heisenberg group $\mathscr{H}$ is defined by
\begin{align}\label{124}
\|(y,{s})\|:=(|y|^{4}+|{s}|^{2})^{\frac{1}{4}}.
\end{align}

Define  balls  $B(\xi,r):=\{\eta\in\mathscr{H} ;\|\xi^{-1}\cdot\eta\|<r\}$
for $\xi\in\mathscr{H}, r>0.$ The fundamental set of $\mathscr{H}$ under the action of the lattice $\mathscr{H}_{\mathbb{Z}}$ in (\ref{dis h}) is
\begin{align}\label{fp}
\mathscr{F}=\{\left.(y,s)
\in\mathscr{H}\right|0\leq y_a<1,0\leq s_\beta<1,a=1,\cdots,4n,\beta=1,2,3\}.
\end{align}
$\mathscr{H}/\mathscr{H}_{\mathbb{Z}}$ is equivalent to $\mathscr{F}$ as a set.
\begin{prop}\label{p2.1}
$\mathscr{H}$ is the disjoint union of $\tau_{(n,m)}\mathscr{F}$ with  $(n,m)\in\mathscr{H}_{\mathbb{Z}}.$
\end{prop}
\begin{proof}
We need to prove that  for any $(y,s)\in\mathscr{H},$ there exist unique  $(y',s')\in\mathscr{F}$ and $(n,m)\in\mathscr{H}_{\mathbb{Z}}$ such that $(y,s)=(n,m)\cdot(y',s').$ Let $ (n_a,m_a)\in\mathscr{H}_{\mathbb{Z}},a=1,2.$
By the multiplication law (\ref{mul}), we have
\begin{equation}\label{eq}
\begin{aligned}
(n_a,m_a)\cdot({y},{s})=\left(n_a+y,(m_a)_\beta+s_{\beta}+2
\sum_{l=0}^{n-1}\sum_{j,k=1}^{4}B_{kj}^{\beta}(n_a)_{4l+k}y_{4l+j}\right).
\end{aligned}
\end{equation}
If $n_1\neq n_2,$ the $y$-coordinates of $(n_1,m_1)\cdot({y},{s})$ and $(n_2,m_2)\cdot({y},{s})$ are $n_1+y$ and $n_2+y,$ respectively,   which are different.  If $n_1= n_2,m_1\neq m_2,$ we see that their $s$-coordinates in (\ref{eq}) must be different. This proves the uniqueness.

For $(y,s)=(y_1,\cdots,y_{4n},s_1,s_2,s_3),$ we can choose $y'\in\mathbb{R}^{4n}$ with $0\leq y_j'<1 $  and $n\in \mathbb{Z}^{4n}$ such that
$y_j={n}_j+y_j' $. Then we can determine $s'\in\mathbb{R}^{3}$ and $m\in \mathbb{Z}^{3}$ satisfying
\begin{equation*}
    m_{\beta}+s'_\beta=s_\beta-2
\sum_{l=0}^{n-1}\sum_{j,k=1}^{4n}B_{kj}^{\beta}n_{4l+k}y'_{4l+j},\ \ {\rm with}\ 0\leq s_\beta'<1,\end{equation*}
for $\beta=1,2,3.$ So $\mathscr{H}$ is the disjoint union of $\tau_{(n,m)}\mathscr{F}.$ The proposition is proved.
\end{proof}
$\mathscr{H}/\mathscr{H}_{\mathbb{Z}}$ has the structure of a locally flat  manifold as follows (cf.  \cite[p. 238]{Kobayashi}). Let $\pi:\mathscr{H}\rightarrow \mathscr{H}/\mathscr{H}_{\mathbb{Z}}$ be the projection. We can find a finite number of balls $B(\xi_j,r),\ j=1,\cdots,N,$
covering $\mathscr{F}$ with $r$ sufficiently small so  that $\tau_{(n,m)} B(\xi_j,r)\cap B(\xi_j,r)=\emptyset$ for any $(0,0)\neq(n,m)\in\mathscr{H}_{\mathbb{Z}}.$ Note that $\pi B(\xi_i,r)\cap \pi B(\xi_j,r)\neq\emptyset$ for $i\neq j$ if and only if there exist unique ${(n,m)}\in\mathscr{H}_{\mathbb{Z}},$ such that
\begin{align}\label{non}
\tau_{(n,m)}B(\xi_i,r)\cap B(\xi_j,r)\neq\emptyset.
\end{align}
Then we can construct coordinates charts $(\pi B(\xi_j,r),\phi_j),$ where $\phi_j:\pi B(\xi_j,r)\rightarrow B(\xi_j,r)$ and the  transition function $\phi_j\circ \phi_i^{-1}$ is given by $\tau_{(n,m)}$ for some $(n,m)\in\mathscr{H}_{\mathbb{Z}}$ such that (\ref{non}) holds.

A function  is called \emph{periodic} on $\mathscr{H}$ if
\begin{align*}
f(y,s)=f((n,m)(y,s))
\end{align*}
for any $(n,m)\in\mathscr{H}_{\mathbb{Z}}.$
A function over $\mathscr{H}/\mathscr{H}_{\mathbb{Z}}$ can be viewed as a function on $\mathscr{F}$ and  be extended   to a periodic function on $\mathscr{H}$ by
\begin{align}\label{3.30}
{f}(y,s)=f((n,m)\cdot(y',s'))=f(y',s'),
\end{align}
for $(y,s)=(n,m)\cdot(y',s')$ and $(y',s')\in\mathscr{F}.$ If $f$ is periodic, then so is $Y_af$ for any $a.$  This is because
\begin{equation*}\begin{aligned}
(Y_{a}f)(y',s')=\left.\frac{\hbox{d}}{\hbox{d}t}f((y',s')(te_{a},0))
\right|_{t=0}=\left.\frac{\hbox{d}}{\hbox{d}t}f((n,m)(y',s')(te_{a},0))
\right|_{t=0}=(Y_{a}f)(y,s),
\end{aligned}\end{equation*}
for $e_{a} $  as in (\ref{eq:left-invariant}). Thus the action of $Y_a$ on functions over  $\mathscr{H}/\mathscr{H}_{\mathbb{Z}}$ is well-defined, i.e. it is a vector field over $\mathscr{H}/\mathscr{H}_{\mathbb{Z}}.$
\subsection{Complex horizontal  vector fields $Z_A^{A'}$'s  and the tangential $k$-Cauchy-Fueter operator}
We consider the following complex horizontal left invariant vector fields on $\mathscr{H} $:
\begin{align}\label{ZAAA}
\left(Z_{AA'}\right):=\left(\begin{array}{ll} Y_{1}+\textbf{i}Y_{2}& -Y_{3}-\textbf{i}Y_{4}\\ Y_{3}-\textbf{i}Y_{4}&\ \  Y_{1}-\textbf{i}Y_{2}\\\ \ \ \ \ \vdots&\ \ \ \ \ \ \ \vdots\\ Y_{4l+1}+\textbf{i}Y_{4l+2}& -Y_{4l+3}-\textbf{i}Y_{4l+4}\\Y_{4l+3}-\textbf{i}Y_{4l+4}& \ \ Y_{4l+1}-\textbf{i}Y_{4l+2}\\ \ \ \ \ \vdots&\ \ \ \ \ \ \ \vdots\end{array}\right)
\end{align}
where $A=0,1,\cdots,2n-1,$ $A'=0',1'.$ It is motivated by the embedding $\tau$ of  quaternionic algebra $\mathbb{H}$ into  $\mathfrak{gl}(2,\mathbb{C}):$
\begin{align}\label{tau}
\tau(x_{1}+x_{2}\textbf{i}+x_{3}\textbf{j}+x_{4}\textbf{k})=
\left(\begin{array}{rr} x_{1}+\textbf{i}x_{2}& -x_{3}-\textbf{i}x_{4}\\ x_{3}-\textbf{i}x_{4}& x_{1}-\textbf{i}x_{2}\end{array}\right)
\end{align}
and vector fields
\begin{align}\label{nabla}
\left(\nabla_{AA'}\right):=\left(\begin{array}{ll} \partial_{x_{1}}+\textbf{i}\partial_{x_{2}}& -\partial_{x_{3}}-\textbf{i}\partial_{x_{4}}\\ \partial_{x_{3}}-\textbf{i}\partial_{x_{4}}&\ \  \partial_{x_{1}}-\textbf{i}\partial_{x_{2}}\\\ \ \ \ \ \vdots&\ \ \ \ \ \ \ \vdots\\ \partial_{x_{4l+1}}+\textbf{i}\partial_{x_{4l+2}}& -\partial_{x_{4l+3}}-\textbf{i}\partial_{x_{4l+4}}\\\partial_{x_{4l+3}}
-\textbf{i}\partial_{x_{4l+4}}& \ \ \partial_{x_{4l+1}}-\textbf{i}\partial_{x_{4l+2}}\\ \ \ \ \ \vdots&\ \ \ \ \ \ \ \vdots\end{array}\right)
\end{align}
to define the $k$-Cauchy-Fueter operators on $\mathbb{H}^{n+1}$ in \cite{wang15}.
We will use  matrices
\begin{align}\label{epsilon}
(\varepsilon_{A'B'})=\left(\begin{array}{cc} 0&1\\ -1&0\end{array}\right),\quad (\varepsilon^{A'B'})=\left(\begin{array}{cc} 0&-1\\ 1&0\end{array}\right)
\end{align}
 to raise or lower primed indices, e.g.
$Z_{A}^{A'}=\sum_{B'=0',1'} Z_{AB'}\varepsilon^{B'A'}.$ Here $(\varepsilon^{A'B'})$ is the inverse of $(\varepsilon_{A'B'}).$ Then $$Z_A^{0'}=Z_{A1'},\quad Z_A^{1'}=-Z_{A0'},$$ and
\begin{align}\label{ZAA}
\left(Z_A^{A'}\right)=\left(\begin{array}{ll} \hskip 3mm \vdots& \hskip 3mm \vdots\\Z_{2l}^{0'}& Z_{2l}^{1'}\\   Z_{2l+1}^{0'}& Z_{2l+1}^{1'}\\ \hskip 3mm  \vdots&  \hskip 3mm \vdots\end{array}\right)
=\left(\begin{array}{ll} \ \ \ \ \ \ \ \vdots&\ \ \ \ \ \ \ \vdots\\-Y_{4l+3}-\textbf{i}Y_{4l+4}& -Y_{4l+1}-\textbf{i}Y_{4l+2}\\\ \ Y_{4l+1}-\textbf{i}Y_{4l+2}& -Y_{4l+3}+\textbf{i}Y_{4l+4}\\ \ \ \ \ \ \ \ \vdots&\ \ \ \ \ \ \ \vdots\end{array}\right).
\end{align}

An element of $\mathbb{C}^{2}$ is denoted by $(f_{A'})$ with $A'=0',1'.$ The symmetric power $\odot^{p}\mathbb{C}^{2}$ is a subspace of $\otimes^{p}\mathbb{C}^2,$ whose element  is  a $2^p$-tuple  $(f_{A'_{1}A'_2\cdots A'_{p}})$ with $A'_{1},A'_2,\cdots,A'_{p}=0',1',$   such that $f_{A'_{1}A'_2\cdots A'_{p}}\in\mathbb{C}$ are invariant under  permutations of subscripts, i.e. $$f_{A'_{1}A'_2\cdots A'_{p}}=f_{A'_{\sigma(1)}A'_{\sigma(2)}\cdots A'_{\sigma(p)}}$$
for any $\sigma$ in the group $S_p$ of permutations  of $p$ letters.  An element of $\odot^{p}\mathbb{C}^2\otimes\wedge^q\mathbb{C}^{2n}$ is given by a tuple $(f_{A_1'\cdots A_p'A_1 \cdots A_q})\in(\otimes^{p}\mathbb{C}^2)\otimes (\otimes^q\mathbb{C}^{2n}),$ which is invariant under permutations of  subscripts of $A_1',\cdots,A_p',$ and  antisymmetric under permutations of   subscripts of $A_1,\cdots,A_q  =0,1,\cdots 2n-1.$ In the sequel, we will write $f_{A A_2'A_3'\cdots A_k'}:=f_{A_2'A_3'\cdots A_k'A}$ and $f_{A_3'\cdots A_k'AB}:=f_{ABA_3'\cdots A_k'}$ for convenience.  We will use {\it symmetrisation} of primed indices
\begin{align}\label{sym}
f_{\cdots(A_1'\cdots A_p')\cdots}:=\frac{1}{p!}\sum_{\sigma\in S_p}f_{\cdots A_{\sigma(1)}'\cdots A_{\sigma(p)}'\cdots}.
\end{align}
The  tangential $k$-Cauchy-Fueter operator  in (\ref{cf}) is given by
\begin{equation}\begin{aligned}\label{zy}
(\mathscr{D}_{0}f)_{AA'_2\cdots A'_k}:=\sum_{A'_1=0',1'}Z_{A}^{A'_1}f_{A'_1 A'_2\cdots A'_k},
\end{aligned}\end{equation}
for $f\in C^1(\Omega,\mathscr{V}_0) $.
The {\it $k$-Cauchy-Fueter operator} on $\mathbb{H}^{n+1}$ \cite{wang15} is $\widehat{\mathscr{D}}_0: C^1(\mathbb{H}^{n+1},\mathscr{V}_0)\rightarrow C^1(\mathbb{H}^{n+1},\mathscr{V}_1)$ with
\begin{align*}
\left(\widehat{\mathscr{D}}_{0}f\right)_{A_2'\cdots A_k'A}:=\sum_{B'=0',1'}\nabla_{A}^{B'}f_{B' A_2'\cdots A_k'},
\end{align*}
where $\nabla$  is given by (\ref{nabla}).
A $\mathscr{V}_0$-valued distribution $f$  is called \emph{$k$-regular} on $\Omega\in\mathbb{H}^{n+1}$ if $\widehat{\mathscr{D}}_0f=0$ on $\Omega$ in the sense of distributions.

\subsection{Commutators of complex horizontal vector fields}
The following nice behavior of commutators of $Z_A^{A'}$'s plays a very important role to show that (\ref{cf}) is a complex and to establish the $L^2$-estimate. It is also the reason why the tangential $k$-Cauchy-Fueter complex on the right Heisenberg group is simpler than that on the left one.
\begin{lem}\label{l3.1}
$(1)$ Vector fields in each column in (\ref{ZAA}) are commutative, i.e. for fixed $A'=0'\ {\rm or}\ 1',$
\begin{align}\label{3.4}
[Z_A^{A'},Z_B^{A'}]=0,
\end{align}for any $A,B=0,\cdots,2n-1.$\\
$(2)$ We have
\begin{equation}\begin{aligned}\label{eq:3.4'}
&[Z_{2l}^{0'},Z_{2l}^{1'}]=
\overline{[Z_{2l+1}^{0'},Z_{2l+1}^{1'}]}
=8\left(\partial_{s_2}+\mathbf{i}\partial_{s_3}\right),\\
&[Z_{2l}^{0'},Z_{2l+1}^{1'}]
=[Z_{2l+1}^{0'},Z_{2l}^{1'}]=8\mathbf{i}\partial_{s_1},
\end{aligned}\end{equation}
$l=0,\cdots,n-1,$ and any other bracket  vanishes.
\end{lem}
\begin{proof}
$(1)$ If $\{A,B\}\neq\{2l,2l+1\}$ for any integer $l,$ we have
\begin{align*}
[Z_A^{A'},Z_B^{B'}]=0,\quad{\rm for}\  A',B'=0',1',
\end{align*}
by using (\ref{2.16'}) because   $Z_A^{A'}$ and $Z_B^{B'}$ only involve $Y_{4l+j}$'s for different  $l.$ It follows from (\ref{2.14}) (\ref{2.16'}) that
\begin{equation}\begin{aligned}\label{com}
\ [Y_{4l+1},Y_{4l+2}]&=\ \ [Y_{4l+3},Y_{4l+4}]=-4\partial_{s_1},\\
[Y_{4l+1},Y_{4l+3}]&=-[Y_{4l+2},Y_{4l+4}]=-4\partial_{s_2},\\
[Y_{4l+1},Y_{4l+4}]&=\ \ [Y_{4l+2},Y_{4l+3}]=-4\partial_{s_3}.
\end{aligned}\end{equation}
Then for $\{A,B\}=\{2l,2l+1\},$ we have
\begin{equation*}\begin{aligned}
\ [Z_{2l}^{0'},Z_{2l+1}^{0'}]=&[-Y_{4l+3}-
\mathbf{i}Y_{4l+4},Y_{4l+1}-\mathbf{i}Y_{4l+2}]\\
=&[Y_{4l+1},Y_{4l+3}]+[Y_{4l+2},Y_{4l+4}]-\mathbf{i}[Y_{4l+2},
Y_{4l+3}]+\mathbf{i}[Y_{4l+1},Y_{4l+4}]=0,\\
\ [Z_{2l}^{1'},Z_{2l+1}^{1'}]=&[-Y_{4l+1}-\mathbf{i}Y_{4l+2},-Y_{4l+3}
+\mathbf{i}Y_{4l+4}]\\
=&[Y_{4l+1},Y_{4l+3}]+[Y_{4l+2},Y_{4l+4}]+\mathbf{i}[Y_{4l+2},
Y_{4l+3}]-\mathbf{i}[Y_{4l+1},Y_{4l+4}]=0,
\end{aligned}\end{equation*}
by (\ref{com}).  Then (\ref{3.4}) follows.\\
$(2)$ Similarly  we have
\begin{equation*}\begin{aligned}
\ [Z_{2l}^{0'},Z_{2l}^{1'}]=&[-Y_{4l+3}-\mathbf{i}
Y_{4l+4},-Y_{4l+1}-\mathbf{i}Y_{4l+2}]
=-[Y_{4l+1},Y_{4l+3}]+[Y_{4l+2},Y_{4l+4}]\\&\quad-\mathbf{i}[Y_{4l+2},
Y_{4l+3}]-\mathbf{i}[Y_{4l+1},Y_{4l+4}]=
8(\partial_{s_{2}}+\mathbf{i}\partial_{s_{3}}),\\
\ [Z_{2l+1}^{0'},Z_{2l+1}^{1'}]=&\overline{[Z_{2l}^{0'},Z_{2l}^{1'}]}
=8(\partial_{s_{2}}-\mathbf{i}\partial_{s_{3}}),\\
\ [Z_{2l}^{0'},Z_{2l+1}^{1'}]=&[-Y_{4l+3}-
\mathbf{i}Y_{4l+4},-Y_{4l+3}+\mathbf{i}Y_{4l+4}]
=-2\mathbf{i}[Y_{4l+3},Y_{4l+4}]
=8\mathbf{i}\partial_{s_{1}},\\
\ [Z_{2l+1}^{0'},Z_{2l}^{1'}]=&[Y_{4l+1}-\mathbf{i}Y_{4l+2},-Y_{4l+1}-
\mathbf{i}Y_{4l+2}]
=-2\mathbf{i}[Y_{4l+1},Y_{4l+2}]
=8\mathbf{i}\partial_{s_{1}},
\end{aligned}\end{equation*}
by (\ref{com}). The lemma is proved.
\end{proof} On the left quaternionic  Heisenberg group, vector fields in each column in (\ref{eq:3.4'}) are not commutative.
We have the following corollary directly by the above Lemma \ref{l3.1} $(2)$.
\begin{cor}\label{cc}
\begin{align}
[Z_{A}^{0'},Z_{B}^{1'}]+[Z_{A}^{1'},Z_{B}^{0'}]=0,
\end{align}for any $A,B=0,\cdots,2n-1.$
\end{cor}
\subsection{The tangential $k$-Cauchy-Fueter complex} Differential operators in the complex (\ref{cf}) are as follows.
For $j=0,1,\cdots,k-1,$   $\mathscr{D}_j:C^{\infty}(\Omega,\mathscr{V}_j)\rightarrow C^{\infty}(\Omega,\mathscr{V}_{j+1})$ with $\mathscr{V}_j=\odot^{k-j}\mathbb{C}^2\otimes\wedge^j\mathbb{C}^{2n} $  is a differential operators of  first order  given by
 \begin{equation}\label{eq:D-j-1}
\left(\mathscr{D}_jf\right)_{A_0\cdots A_jA_1'\cdots A_{k-j-1}'}=(j+1)\sum_{A'=0',1'}Z_{[A_0}^{A'}f_{A_1\cdots A_j] A'A_1'\cdots A_{k-j-1}'},
 \end{equation}
where $[A_0A_1\cdots A_j]$ is the {\it antisymmetrisation} of unprimed indices  given by
\begin{align}\label{antisym}
f_{\cdots[A_1\cdots A_p]\cdots}:=\frac{1}{p!}\sum_{\sigma\in S_p}{\rm sign}(\sigma)f_{\cdots A_{\sigma(1)}\cdots A_{\sigma(p)}\cdots}.
\end{align}In particular,  $h_{[AB]}:=\frac{1}{2}(h_{AB}-h_{BA})$.  By definition, we have
\begin{align}\label{bra}
f_{\cdots[A_1\cdots[A_j\cdots A_l]\cdots A_p]\cdots}=f_{\cdots[A_1\cdots A_j\cdots A_l \cdots A_p]\cdots}.
\end{align}  $\mathscr{D}_k:C^{\infty}(\Omega,\mathscr{V}_k)\rightarrow C^{\infty}(\Omega,\mathscr{V}_{k+1})$ with $\mathscr{V}_k=\wedge^k\mathbb{C}^{2n}$ and $\mathscr{V}_{k+1}=\wedge^{k+2}\mathbb{C}^{2n}$  is a differential operator  of  second order given by
 \begin{equation}\label{eq:D-j-2}\left(\mathscr{D}_{k}f\right)_{A_1\cdots A_{k+2}}=(k+2)Z_{[A_1}^{0'}Z_{A_2}^{1'}f_{A_3\cdots A_{{k+2}}]}.
 \end{equation}
For $j=k+1,\cdots,2n-2,$ $\mathscr{D}_j:C^{\infty}(\Omega,\mathscr{V}_j)\rightarrow C^{\infty}(\Omega,\mathscr{V}_{j+1})$ with $\mathscr{V}_j=\odot^{j-k-1}\mathbb{C}^2\otimes\wedge^{j+1}\mathbb{C}^{2n} $  is a differential operator  of  first order  given by
 \begin{equation}\label{eq:D-j-3}
    \left(\mathscr{D}_{j}f\right)_{A_1\cdots A_{j+2}}^{A_1'\cdots A_{j-k}'}=(j+2)Z_{[A_1}^{(A_1'}f^{A_2'\cdots A'_{{j-k}})}_{A_2\cdots A_{{j+2}}]}.
 \end{equation}

\begin{rem}\label{rem:kk}
  The   $k$-Cauchy-Fueter complex on $\mathbb{H}^n$ \cite{wang15,wang24} is the same as (\ref{cf})-(\ref{V}) with $\mathscr H$ replaced by $\mathbb{H}^n$  and $Z_{A}^{A'}$ in definition  of $\mathscr{D}_{j}$'s  in  (\ref{eq:D-j-1}) (\ref{eq:D-j-2})  (\ref{eq:D-j-3}) replaced by $\nabla_{A}^{A'}$ in  (\ref{nabla}).
\end{rem}

\begin{lem}\label{ll}
\begin{align}
Z_{[A}^{(A'}Z_{B]}^{B')}=0,
\end{align}
for any $A,B=0,\cdots,2n-1$ and $A',B'=0',1'.$
\end{lem}
\begin{proof}
Note that \begin{align}\label{1}
2Z_{[A}^{A'}Z_{B]}^{A'}=Z_{A}^{A'}Z_{B}^{A'}-Z_{B}^{A'}Z_{A}^{A'}=[Z_{A}^{A'},
Z_{B}^{A'}]=0,
\end{align}
by (\ref{3.4}), and
\begin{equation*}\begin{split}4
Z_{[A}^{(0'}Z_{B]}^{1')}&=
2Z_{[A}^{0'}Z_{B]}^{1'}+2Z_{[A}^{1'}Z_{B]}^{0'}=Z_{A}^{0'}Z_{B}^{1'}-
Z_{B}^{0'}Z_{A}^{1'}+Z_{A}^{1'}Z_{B}^{0'}-
Z_{B}^{1'}Z_{A}^{0'}\\&=[Z_{A}^{0'},Z_{B}^{1'}]+[Z_{A}^{1'},Z_{B}^{0'}]=0,
\end{split}\end{equation*}
 by Corollary \ref{cc}. The lemma is proved.
\end{proof}
Now let us  check (\ref{cf}) to be a complex by direct calculation as in \cite[Section 3.1]{wang24}.
\begin{thm}
$(\ref{cf})$ is a complex, i.e.
\begin{align}\label{d10}
\mathscr{D}_{j+1}\circ\mathscr{D}_j=0
\end{align}  for each $j$.
\end{thm}
\begin{proof}

For $A,B=0,\cdots, 2n-1$ and $A_3',\cdots, A_k'=0',1',$ we have
\begin{equation*}
\begin{aligned}
(\mathscr{D}_1\circ\mathscr{D}_0f)_{ABA_3'\cdots A_k'}=&2\sum_{A'=0',1'}
Z_{[A}^{A'}(\mathscr{D}_0f)_{B]A'A_3'\cdots A_k'}
=2\sum_{A',C'=0',1'}
Z_{[A}^{A'}Z_{B]}^{C'}f_{C'A'A_3'\cdots A_k'}\\
=&2\sum_{A',C'=0',1'}
Z_{[A}^{(A'}Z_{B]}^{C')}f_{C'A'A_3'\cdots A_k'}=0,
\end{aligned}\end{equation*}
by Lemma \ref{ll} and $f_{C'A'A_3'\cdots A_k'}=f_{A'C'A_3'\cdots A_k'}$.
For general $j=1,\cdots,k-2,$ we have
\begin{equation*}
\begin{aligned}
(\mathscr{D}_{j+1}\circ\mathscr{D}_jf)_{A_1\cdots A_{j+2}A_1'\cdots A_{k-j-2}'}=&(j+2)(j+1)\sum_{A',C'=0',1'}
Z_{[A_1}^{A'}Z_{[A_2}^{C'}f_{A_3\cdots A_{j+2}]]C'A'A_1'\cdots A_{k-j-2}'} \\
=& (j+2)(j+1) \sum_{A',C'=0',1'}
Z_{[[A_1}^{(A'}Z_{A_2]}^{C')}f_{A_3\cdots A_{j+2}]C'A'A_1'\cdots A_{k-j-2}'}=0,
\end{aligned}\end{equation*}
by using (\ref{bra})  repeatedly, Lemma \ref{ll} and $ f$ symmetric in the primed indices again.

For $j=k-1,$ we have
\begin{equation*}
\begin{aligned}
(\mathscr{D}_{k}\circ\mathscr{D}_{k-1}f)_{A_1\cdots A_{k+2}}=&(k+2)k\sum_{A'=0',1'}
Z_{[A_1}^{0'}Z_{A_2}^{1'}Z_{[A_3}^{A'}f_{A_4\cdots A_{{k+2}}]]A'}=0.
\end{aligned}\end{equation*}
This is because if $A'=1',$    $Z_{[A_1}^{0'}Z_{[A_2}^{1'}Z_{A_3]}^{1'}f_{A_4\cdots A_{{(k+2)}}]1'}=0$ by using
  (\ref{1}), and  if $A'=0',$
  \begin{equation*}
     Z_{[A_1}^{0'}Z_{A_2}^{1'}Z_{A_3]}^{0'}=Z_{[A_1}^{0'}Z_{[A_2}^{1'}Z_{A_3]]}^{0'}=
-Z_{[A_1}^{0'}Z_{[A_2}^{0'}Z_{A_3]]}^{1'}=
-Z_{[[A_1}^{0'}Z_{A_2]}^{0'}Z_{A_3]}^{1'}=0,
  \end{equation*}
  by using
  (\ref{bra}) repeatedly and Corollary \ref{cc}.

For  $j=k,$ we have
\begin{align}
\left(\mathscr{D}_{k+1}\circ\mathscr{D}_kf\right)_{A_1\cdots A_{k+3}}^{A'}=(k+3)(k+2)Z_{[A_1}^{A'}Z_{[A_2}^{0'}Z_{A_3}^{1'}f_{A_4\cdots A_{{k+3}}]]}=0.
\end{align}
This is because if $A'=0',$    $Z_{[[A_1}^{0'}Z_{A_2]}^{0'}Z_{A_3}^{1'}f_{A_4\cdots A_{{k+3}}]}=0$ by using   (\ref{1}), and if $A'=1',$
\begin{equation*}
   Z_{[A_1}^{1'}Z_{A_2}^{0'}Z_{A_3]}^{1'}=Z_{[A_1}^{1'}Z_{[A_2}^{0'}Z_{A_3]]}^{1'}=
-Z_{[A_1}^{1'}Z_{[A_2}^{1'}Z_{A_3]]}^{0'}=
-Z_{[[A_1}^{1'}Z_{A_2]}^{1'}Z_{A_3]}^{0'}=0,
\end{equation*}
  by using
  (\ref{bra}) repeatedly and Corollary \ref{cc}.

For $j=k+1,\cdots,2n-2,$ we have
$$\left(\mathscr{D}_{j+1}\circ\mathscr{D}_jf\right)_{A_1\cdots A_{j+3}}^{A_1'\cdots A'_{j-k+1}}=(j+3)(j+2)Z^{((A_1'}_{[[A_1}Z_{A_2]}^{A_2')}f^{A_3'\cdots A'_{{j-k+1}})}_{A_3\cdots A_{{j+3}}]}=0,$$ by Lemma \ref{ll}.  The theorem is proved.
\end{proof}

\subsection{Comparison with the left case}\label{sub25}
Recall that a transformation $T$ on $\mathscr{H}$ is called {\it conformal} if $\|T_*W_1\|=\|T_*W_2\|$  for any two horizontal vector fields $W_1$ and $W_2$ with  $\|W_1\|=\|W_2\|,$ where $\|W\|^2:= \sum_{j=1}^{4n}a_j^2 $ if we write $W=\sum_{j=1}^{4n}a_jY_j.$ It is known that the group  of conformal transformations on $\mathscr{H}$ is ${\rm Sp}(n+1,1)$ (cf. e.g. \cite{Ivanov}) generated by the following transformations:\\
(1) \emph{dilations}:
\begin{align}
D_{\delta}:(y,s)\longrightarrow(\delta y,\delta^{2}s),\ \delta>0;
\end{align}
(2) \emph{left translations}:
\begin{align}\label{2.25}
\tau_{(x,{t})}:(y,{s})\longrightarrow
(x,{t})\cdot(y,{s});
\end{align}
(3) \emph{rotations}:
\begin{align}\label{33}
R_\mathbf a:(y,s)\longrightarrow (y\mathbf a,s),\ {\rm for} \ \mathbf a\in {\rm Sp}(n),
\end{align}
where
\begin{equation*}
  {\rm Sp}(n)=\{\mathbf a\in {\rm GL}(n,\mathbb{H})|{\mathbf a\bar{\mathbf a}^{t}}=I_{n}\};
\end{equation*}
(4) The \emph{inversion}:
\begin{align}\label{44}
R:(y,s)\longrightarrow \left(-(|y|^{2}-s)^{-1}y,
\frac{-s}{|y|^{4}+|s|^{2}}\right);
\end{align}\\
(5) ${\rm Sp}(1)$ acts on $\mathscr{H}$ as
\begin{align}\label{66}
\sigma:(y,s)\longrightarrow (\sigma y,\sigma{s}\sigma^{-1}),
\end{align}
where the action on the first factor is left multiplication by $\sigma\in\mathbb{H}$ with $|\sigma|=1,$ while the action on the second factor  is isomorphism with ${\rm SO}(3)$.

The multiplication (\ref{hei'}) of the left quaternionic Heisenberg group $\widetilde{\mathscr H}$ can be written as  \begin{align}\label{mul'}
(x,t)\cdot({y},{s})= {\left(x+y,t_{\beta}+s_{\beta}+2
\sum_{l=0}^{n-1}\sum_{j,k=1}^{4}I_{kj}^{\beta}x_{4l+k}y_{4l+j}\right)},
\end{align}
for $x,y\in \mathbb{R}^{4n},\ t,s\in \mathbb{R}^{3},\ \beta=1,2,3,$ where $I_{kj}^{\beta}$ is the $(k,j)$-th entry of the following matrices
\begin{equation}\label{I}
\begin{aligned}
I^{1}:=\left(\begin{array}{cccc} 0 & 1 & 0 &0\\ -1& 0& 0& 0\\ 0& 0&0& -1\\0 &0& 1 &0\end{array}\right),    I^{2}:=\left(\begin{array}{cccc} 0 & 0 &
1 &0\\ 0& 0& 0& 1\\ -1& 0&0& 0\\0 &-1& 0 &0\end{array}\right),
I^{3}:=\left(\begin{array}{cccc} 0 & 0 & 0 &1\\ 0& 0& -1& 0\\ 0& 1&0& 0\\-1 &0& 0 &0\end{array}\right),
\end{aligned}
\end{equation}
satisfying the commutating relation of quaternions.

Recall that  $\mathfrak{g}=\mathfrak{sp}(2(n+2),\mathbb{C})$  has the decomposition $\mathfrak{g}=\mathfrak{g}_{-2}\oplus\mathfrak{g}_{-1}\oplus\mathfrak{g}_{0}
\oplus\mathfrak{g}_{1}\oplus\mathfrak{g}_{2},$ where $\mathfrak{g}_{-2}$ is an complex  abelian subalgebra generated by $T_1,T_2,T_3,$ and  $\mathfrak{g}_{-1}$ is generated by $\{Y_{AA'}\},A=0,1,\cdots,2n-1,A'=0',1'$ with
\begin{equation}\begin{aligned}
\ [Y_{A0'},Y_{(n+A)0'}]=&4T_2,\\ \ [Y_{A1'},Y_{(n+A)1'}]=&4T_3,\\ \ [Y_{A0'},Y_{(n+A)1'}]=&[Y_{A1'},Y_{(n+A)0'}]=4T_1,
\end{aligned}\end{equation}and any other bracket  vanishes
(cf. \cite[(2.10)]{wang1}). $\mathfrak p:=\mathfrak{g}_{0}
\oplus\mathfrak{g}_{1}\oplus\mathfrak{g}_{2}$ is a parabolic subgroup. $\mathfrak{u}_-:=\mathfrak{g}_{-2}\oplus\mathfrak{g}_{-1}.$ Let ${\rm U}_-$ be the complex Lie group with Lie algebra $\mathfrak{u}_-$. Then
\begin{align}\mathfrak{g}=\mathfrak{u}_-\oplus\mathfrak p.\end{align}
There exist exact sequences \cite[Theorem 3.2.1]{wang1} on ${\rm U}_-$ \begin{equation}\begin{aligned}\label{cf'}
0&\rightarrow \mathcal{R}\left({\rm U}_-,\odot^{k}\mathbb{C}^2\right)
\xrightarrow{Q_{0}^{(k)}} \mathcal{R}\left({\rm U}_-,\odot^{k-1}\mathbb{C}^2\otimes{V}^{(1)}\right)
\xrightarrow{Q_{1}^{(k)}}\cdots\rightarrow \mathcal{R}\left({\rm U}_-,{V}^{(k)}\right)\\ &\xrightarrow{Q_{k}^{(k)}} \mathcal{R}\left({\rm U}_-,{V}^{(k+2)}\right)\xrightarrow{Q_{k+1}^{(k)}}\cdots
\xrightarrow{Q_{2n-1}^{(k)}}\mathcal{R}\left({\rm U}_-,\odot^{2n-k}
\mathbb{C}^2\right)\rightarrow0,
\end{aligned}\end{equation}
for $0\leq k\leq n-2,$ where $Q_{j}^{(k)}$ is defined in terms of $Y_{AA'},T_\beta$ (cf. \cite[Thoerem 1.0.1]{wang1}). Here  $V^{(j)}$ is the irreducible representation of $\mathfrak{sp}(2n,\mathbb{C})$ with the highest weight to be the $j$-th fundamental weight $\omega_j$ and $\mathcal{R}( {\rm U}_-,V)$ is the ring of $V$-valued polynomials over ${\rm U}_-.$
These complexes are constructed by twistor method, and  operators $Q_{j}^{(k)}$'s are invariant under $\mathfrak{sp}(2(n+2),\mathbb{C}).$ Note that
\begin{align}\label{2.43'}
\widetilde{X}_{4l+j}=\frac{\partial}{\partial x_{4l+j}}+2\sum_{\beta=1}^{3}\sum_{k=1}^{4}I^{\beta}_{kj}x_{4l+k}
\frac{\partial}{\partial t_{\beta}}
\end{align}
is standard left invariant vector field  on $\widetilde{\mathscr{H}}.$
Define \begin{align*}
\left(\widetilde{Z}_{AA'}\right):=\left(\begin{array}{ll} \widetilde{X}_{1}+\mathbf{i}\widetilde{X}_{2}& -\widetilde{X}_{3}-\mathbf{i}\widetilde{X}_{4}\\ \widetilde{X}_{3}-\mathbf{i}\widetilde{X}_{4}&\ \  \widetilde{X}_{1}-\mathbf{i}\widetilde{X}_{2}\\\ \ \ \ \ \vdots&\ \ \ \ \ \ \ \vdots\\ \widetilde{X}_{4l+1}+\mathbf{i}\widetilde{X}_{4l+2}& -\widetilde{X}_{4l+3}-\mathbf{i}\widetilde{X}_{4l+4}\\\widetilde{X}_{4l+3}-\mathbf{i}\widetilde{X}_{4l+4}& \ \ \widetilde{X}_{4l+1}-\mathbf{i}\widetilde{X}_{4l+2}\\ \ \ \ \ \ \vdots&\ \ \ \ \ \ \ \vdots\end{array}\right),
\end{align*}
where $A=0,1,\cdots,2n-1,$ $A'=0',1'.$ They satisfy  the following commutating relations:
\begin{equation}\begin{aligned}\label{3.4'}
&\left[\widetilde{Z}_{(2l)0'},\widetilde{Z}_{(2l+1)0'}\right]
=8\left(\partial_{t_2}-\mathbf{i}\partial_{t_3}\right),\\
&\left[\widetilde{Z}_{(2l)1'},\widetilde{Z}_{(2l+1)1'}\right]
=8\left(\partial_{t_2}+\mathbf{i}\partial_{t_3}\right),\\
&\left[\widetilde{Z}_{(2l)0'},\widetilde{Z}_{(2l+1)1'}\right]
=\left[\widetilde{Z}_{(2l)1'},\widetilde{Z}_{(2l+1)0'}\right]=-8\mathbf{i}\partial_{t_1},
\end{aligned}\end{equation}
$l=0,\cdots,n-1,$ and any other bracket  vanishes.
So by embedding the real Lie algebra of $\widetilde{\mathscr H}$ into the complex Lie algebra $\mathfrak{u}_-$ by $\widetilde Z_{AA'}\mapsto Y_{AA'},$ we get tangential $k$-Cauchy-Fueter  complexes on $\widetilde{\mathscr H}$ (cf. \cite[Thoerem 1.0.1]{wang1}), on which $G={\rm Sp}(2(n+2),\mathbb{C})$ acts naturally.

Now consider complexes on the right quaternionic Heisenberg group. We can show the following proposition as \cite[Proposition 3.1]{wang24'}.
\begin{prop}\label{pp2.2}
Under the transformation $M_{\mathbf{a}}:\mathbb H^n\rightarrow\mathbb H^n,$ $q\mapsto q'=q\mathbf a$ with  $\mathbf{a}=(a_{jk})\in GL(n,\mathbb{H}),$ where $q=(q_1,q_2,\cdots,q_{n})$ with $q_{l}={x_{4l-3}}+\mathbf{i}{x_{4l-2}}
+\mathbf j{x_{4l-1}}+\mathbf k{x_{4l}},$ we have
\begin{align}\label{par}
\overline{\partial}_{q_l}\left[f(q\mathbf a)\right]=\sum_{m=1}^n\left[\overline{\partial}_{q'_m} (\bar{\mathbf{a}}_{l m} f) \right](q\mathbf a),
\end{align}
where $\overline{\partial}_{q_{l}}=\partial_{x_{4l-3}}+\mathbf{i}\partial_{x_{4l-2}}
+\mathbf{j}\partial_{x_{4l-1}}+\mathbf{k}\partial_{x_{4l}}.$
\end{prop}
\begin{proof}
Denote $\widehat q=(x_1,\cdots,x_{4n}) $. Since $M_{\mathbf{a}}$ define a real linear transformation on the underlying vector space $\mathbb{R}^{4n}$, we have $\widehat{q\mathbf a}=\widehat q\mathbf a^\mathbb{R} $
for some $(4n)\times(4n)$ real matrix $\mathbf a^\mathbb{R}$ associated to $\mathbf a.$ As the $b$-th element of $\widehat{q\mathbf a}$ is $\sum_{a=1}^{4n}x_a\mathbf a^\mathbb{R}_{a b},$ we have
\begin{align*}
\frac{\partial}{\partial x_a}\left[f(q\mathbf a)\right]=\sum_{b=1}^{4n}\frac{\partial f}{\partial x_b}(q\mathbf a) \mathbf a^\mathbb{R}_{a b}.
\end{align*}Note that we can write $q =\sum_{l=1}^{4n}\sum_{j=1}^4\mathbf i_{j-1}  x_{4l+j} $.
Therefore,
\begin{equation*}\begin{aligned}
M_{\mathbf{a}*}\overline{\partial}_{q_{l+1}}=&\sum_{j=1}^4\mathbf i_{j-1}M_{\mathbf{a}*}\frac{\partial}{\partial x_{4l +j}}=\sum_{b=1}^{4n}\sum_{j=1}^4\mathbf i_{j-1}\frac{\partial}{\partial x_{b}}\mathbf a^\mathbb{R}_{(4l +j)b}\\=&\sum_{b=1}^{4n}\sum_{j=1}^4\mathbf i_{j-1}\frac{\partial}{\partial x_{b}}\left( {\mathbf a}^\mathbb{R}\right)^t_{b(4l +j)}=\sum_{m=1}^{ n}
\overline{\partial}_{q'_m}\cdot\bar{\mathbf{a}}_{l m},
\end{aligned}\end{equation*}by  $\left( {\mathbf a}^\mathbb{R}\right)^t=\left( \overline{{\mathbf a}}^t\right)^\mathbb{R}$.
The proposition is proved.
\end{proof}
\begin{cor}
${R_\mathbf{a}}_*\left(\overline{Q}_1,\cdots,\overline{Q}_n\right)
=\left(\overline{Q}_1,\cdots,\overline{Q}_n\right)\bar{\mathbf a}^t,$ for $\mathbf a\in{\rm Sp}(n),$ where  $\overline{Q}_{l+1}=X_{4l+1}+\mathbf iX_{4l+2}+\mathbf jX_{4l+3}+\mathbf kX_{4l+4}.$
\end{cor}
Since  $\overline{Q}_l=\overline{\partial}_{q_l}$ at the origin of $\mathscr H,$ the above identity holds at the origin by Proposition \ref{pp2.2}. It holds at other place by the left invariance.
By applying the  representation $\tau$ in (\ref{tau}),
i.e. $\tau (q_1q_2)=\tau (q_1)\tau (q_2)$
for any  $q_1,q_2\in\mathbb{H} $ {(cf.  \cite[Proposition 2.1]{Wan})},  we get
{\small\begin{equation}\begin{aligned}\label{pre}
&{R_\mathbf{a}}_*\left(\begin{array}{llllll}Z_{00'}Z_{01'}&\cdots& Z_{(2l)0'}&Z_{(2l)1'}&\cdots\\Z_{10'}&Z_{11'}&\cdots & Z_{(2l+1)0'}&Z_{(2l+1)1'}&\cdots\end{array}\right)\\=&\left(\begin{array}{llllll}Z_{00'}&Z_{01'}&\cdots & Z_{(2l)0'}&Z_{(2l)1'}&\cdots\\Z_{10'}&Z_{11'}&\cdots & Z_{(2l+1)0'}&Z_{(2l+1)1'}&\cdots\end{array}\right)\tau(\bar{\mathbf a}^t),
\end{aligned}\end{equation}}for  rotation ${R_\mathbf{a}}$ in (\ref{33}) with $\mathbf a\in{\rm Sp}(n),$
where $\tau(\bar{\mathbf a}^t)$ is a $(2n)\times(2n)$ complex matrix with $\overline{\mathbf{a}}_{jk}$ replaced by the $2\times2$ matrix $\tau(\overline{\mathbf{a}}_{jk}).$
(\ref{pre}) implies that the abelian subalgebras of each column in (\ref{nabla}) is not invariant under the rotations (\ref{33}) of ${\rm Sp}(n).$ Note that commutativity (\ref{3.4}) of each column   plays an very  important role in the construction of our complexes (\ref{cf}).  So by definition (\ref{eq:D-j-1}) (\ref{eq:D-j-2}) and (\ref{eq:D-j-3}), the differential operators $\mathscr D_j$'s in the complex (\ref{cf}) in terms of $Z_{A}^{A'}$'s are not invariant under ${\rm Sp}(n).$ Therefore they are not invariant under ${\rm Sp}(2(n+2),\mathbb{C}).$

Another difference is that the kernel of the tangential $k$-Cauchy-Fueter in space of $L^2$ integrable  function on the left quaternionic Heisenberg group is infinite dimensional
\cite{Shi2}, while it is trivial on the right quaternionic Heisenberg group, since such a function satisfies $\Delta_b f=0$ by \cite[Proposition 2.4]{Shi2} and $\ker  \Delta_b=\{0\}$ in the $L^2$ space.

On the other hand if we define the complex horizontal fields on $\mathscr H$
\begin{equation}\begin{aligned}
\left(\widehat{Z}_{AA'}\right):=\left(\begin{array}{ll} -Y_{1}+\mathbf{i}Y_{2}& -Y_{3}-\mathbf{i}Y_{4}\\\ \ Y_{3}-\mathbf{i}Y_{4}&  -Y_{1}-\mathbf{i}Y_{2}\\\ \ \ \ \ \ \vdots&\ \ \ \ \ \ \ \vdots\\ -Y_{4l+1}+\mathbf{i}Y_{4l+2}& -Y_{4l+3}-\mathbf{i}Y_{4l+4}\\ \ \ Y_{4l+3}-\mathbf{i}Y_{4l+4}& -Y_{4l+1}-\mathbf{i}Y_{4l+2}\\ \ \ \ \ \ \ \vdots&\ \ \ \ \ \ \ \vdots\end{array}\right)
\end{aligned}\end{equation}
with $Y_{4l+1}$   replaced by $-Y_{4l+1},$ then $\widehat{Z}_{AA'}$'s satisfy
\begin{equation*}\begin{aligned}
&\left[\widehat{Z}_{(2l)0'},\widehat{Z}_{(2l+1)0'}\right]
=8\left(\partial_{t_2}-\mathbf{i}\partial_{t_3}\right),\\
&\left[\widehat{Z}_{(2l)1'},\widehat{Z}_{(2l+1)1'}\right]
=8\left(\partial_{t_2}+\mathbf{i}\partial_{t_3}\right),\\
&\left[\widehat{Z}_{(2l)0'},\widehat{Z}_{(2l+1)1'}\right]
=\left[\widehat{Z}_{(2l)1'},\widehat{Z}_{(2l+1)0'}\right]=-8\mathbf{i}\partial_{t_1},
\end{aligned}\end{equation*}
$l=0,\cdots,n-1,$ and any other bracket  vanishes,
i.e. we can embed the real Lie algebra of $\mathscr H$ into the complex Lie algebra ${\rm U}_-.$ Then the complexes (\ref{cf'}) on ${\rm U}_-$ induces a family of  complexes on $\mathscr{H}$ invariant under ${\rm Sp}((2n+2),\mathbb{C}).$ But the first operator is different from the first one in (\ref{cf}). Moreover the $(n-1)$-th operator in the complex induced from ${\rm U}_-$ is a linear combination of $T_\beta$'s (cf. \cite[Proposition 4.3.3]{wang1}),  while the $(n-1)$-th operator in (\ref{cf}) involves only $Z_{AA'}$'s. They are different complexes.
Changing $Y_{4l+1}$ to $-Y_{4l+1},$ corresponds to changing  the sign before $x_{4l+1}^2$ in the definition (\ref{hyper}) of the hypersurface $\mathcal{S}$. The resulting hypersurface is essentially the boundary of the quaternionic Siegel domain.

On other quadratic hypersurface, there is no reason to expect that the restriction of the $k$-Cauchy-Fueter operators and complexes are   invariant in general under the action of ${\rm Sp}(n).$

\subsection{The adjoint operator}
On a domain $\Omega\subset\mathscr{H},$ denote the inner product
\begin{align*}
(u,v):=\int_{\Omega}u\cdot \overline{v}\hbox{d}V,
\end{align*}
for $u,v\in L^2(\Omega,\mathbb{C}),$ where $\hbox{d}V$ is the  Lebesgue measure on $\mathscr{H}.$ The inner product of $L^2(\Omega,\mathscr{V}_1)$ is defined as
$$\langle f,h\rangle:=\sum_{A=0}^{2n-1}
\sum_{A_2',\cdots,A_k'=0',1'}\left(f_{AA_2'\cdots A_k'},h_{AA_2'\cdots A_k'}\right)$$
for $f,h\in L^2(\Omega,\mathscr{V}_1),$ and $\|f\|:=\langle f,f\rangle^{\frac{1}{2}}.$ We define inner products of   $L^2(\Omega,\mathscr{V}_0)$ and $L^2(\Omega,\mathscr{V}_2)$  similarly. Define the $L^2$-norm on $\mathscr{H}/\mathscr{H}_{\mathbb{Z}}$ by
\begin{align*}
\|f\|^2_{L^2(\mathscr{H}/\mathscr{H}_{\mathbb{Z}})}=
\|f\|^2_{L^2(\mathscr{F})}=\int_{\mathscr{F}}|f|^2\hbox{d}V.
\end{align*}
\begin{prop}
The formal adjoint operator of $Z_A^{A'}$ is
\begin{align}\label{delta}
\left(Z_A^{A'}\right)^*=\delta_{A'}^A,\quad {\rm where}\quad \delta^A_{A'}:=-\overline{Z_A^{A'}}.
\end{align}
\end{prop}
\begin{proof}
For $u,v\in C_0^\infty(\mathscr{H},\mathbb{C}),$ we have
$$(Y_au,v)=(u,-Y_av)$$
by integration by part. So
$((Y_a\pm\mathbf{i}Y_b)u,v)=(u,-(Y_a\mp\mathbf{i}Y_b)v).$
Then (\ref{delta}) holds since $Z_A^{A'}$ has the form $Y_a\pm\mathbf{i}Y_b$ for some $a$ and $b$ by (\ref{ZAA}).
Thus we have
\begin{align}\label{bypart}
\left(Z_A^{A'}u,v\right)=\left(u,\delta^A_{A'}v\right)
\end{align}
over $\mathscr{H}.$ For (\ref{bypart}) over $\mathscr{H}/\mathscr{H}_{\mathbb{Z}},$ by using the unit partition, it is sufficient to show it   for $v\in C_0^\infty(\mathscr{H},\mathbb{C}).$ This case follows from the  result over $\mathscr{H}.$
\end{proof}

\begin{lem}
For $f\in C_0^1(\mathscr{H},\mathscr{V}_1)$ or $C^1(\mathscr{H}/\mathscr{H}_{\mathbb{Z}},\mathscr{V}_1),$ we have
\begin{align}\label{dstar}
\left(\mathscr{D}_0^*f\right)_{A_1'\cdots A_k'}=\sum_{A=0}^{2n-1}\delta^A_{(A_1'}f_{A_2'\cdots A_k')A}.
\end{align}
\end{lem}
\begin{proof}
The proof is  similar to that for the $k$-Cauchy-Fueter operator over $\mathbb{H}^n$ (cf.  \cite[Lemma 3.1]{wang23}).
For any $g\in C^1(\mathscr{H}/\mathscr{H}_{\mathbb{Z}},\mathscr{V}_0),$ we have
\begin{equation*}\begin{aligned}
\langle\mathscr{D}_0g,f\rangle=&\sum_{A,A_2',\cdots,A_k'}\left(\sum_{A_1'}
Z_A^{A_1'}g_{A_1'\cdots A_k'},f_{A_2'\cdots A_k'A}\right)=\sum_{A,A_1',\cdots,A_k'}\left(
g_{A_1'\cdots A_k'},\delta_{A_1'}^Af_{A_2'\cdots A_k'A}\right)\\=&\sum_{A_1',\cdots,A_k'}\left(
g_{A_1'\cdots A_k'},\sum_A\delta^A_{(A_1'}f_{A_2'\cdots A_k')A}\right)=\langle g,\mathscr{D}_0^*f\rangle
\end{aligned}\end{equation*}
by using (\ref{bypart}) and symmetrisation
$$\sum_{A_1',\cdots,A_k'}\left(
g_{A_1'\cdots A_k'},G_{A_1'\cdots A_k'}\right)=\sum_{A_1',\cdots,A_k'}\left(
g_{A_1'\cdots A_k'},G_{(A_1'\cdots A_k')}\right)$$
for any $g\in L^2(\mathscr{H},\odot^k\mathbb{C}^2),G\in L^2(\mathscr{H},\otimes^k\mathbb{C}^2).$ (cf.  \cite[(3.4)]{wang23}). Here we have to symmetrise the primed indices  in $\sum_A\delta^A_{A_1'}f_{A_2'\cdots A_k'A}$ since only after symmetrisation it becomes an element of $C_0^1(\mathscr{H},\mathscr{V}_0).$
\end{proof}

$\mathscr{D}_0^*\mathscr{D}_0$ is simple since it is diagonal by the following proposition.
\begin{prop}\label{DD} For $f\in C^2(\Omega,\mathscr{V}_0),$ we have
$$\mathscr{D}_0^*\mathscr{D}_0f=\Delta_bf.$$
\end{prop}
\begin{proof}
Recall that for a $\otimes^k\mathbb{C}^2$-valued function $F_{A_1'\cdots A_k'}$   symmetric in $A_2'\cdots A_k',$  we have
\begin{align}\label{2.35}
F_{(A_1'\cdots A_k')}=\frac{1}{k}\left(F_{A_1'A_2'\cdots A_k'}+\cdots+F_{A_s'A_1'\cdots \widehat{A_s'}\cdots A_k'}+\cdots+F_{A_k'A_1'\cdots \widehat{A_k'}}\right),
\end{align}
by the  definition of symmetrisation (\ref{sym}). As usual, a hat means omittance of the corresponding index.
Then for fixed $A'_1,\cdots,A'_k=0',1',$
\begin{equation}\begin{aligned}\label{dd}
\left(\mathscr{D}_0^*\mathscr{D}_0f\right)_{A_1'\cdots A_k'}=&\sum_A\delta^A_{(A_1'}
\left(\mathscr{D}_0f\right)_{A_2'\cdots A_k')A}=\frac{1}{k}\sum_{s=1}^k
\delta^A_{A_s'}\left(\mathscr{D}_0f\right)_{\cdots \widehat{A_s'}\cdots A_k'A}\\
=&-\frac{1}{k}\sum_{s=1}^k\sum_{A,A'}\overline{Z_A^{A'_{s}}}
Z_A^{A'}f_{A'\cdots \widehat{A_s'}\cdots A_k'}\\=&\frac{1}{k}\sum_{s=1}^k\sum_{A'}\Delta_bf_{A' \cdots \widehat{A_s'}\cdots A_k'}\delta_{A_s'A'}=\Delta_bf_{A_1'\cdots A_k'},
\end{aligned}\end{equation}
by using the following Lemma \ref{l3.2} and $f$   symmetric in the primed indices, where $\mathscr{D}_0^*$ is given by (\ref{dstar}).
The proposition is proved.
\end{proof}
\begin{lem}\label{l3.2}
For $A',B'=0',1',$  we have
\begin{align}\label{3.9}
\sum_{A=0}^{2n-1}\overline{Z_A^{A'}}Z_A^{B'}=-\delta_{A'B'}\Delta_b.
\end{align}
\end{lem}
\begin{proof}
Note that
\begin{equation*}\begin{aligned}
\overline{Z_{2l}^{0'}}Z_{2l}^{0'}+
\overline{Z_{2l+1}^{0'}}Z_{2l+1}^{0'}&=
(-Y_{4l+3}+\mathbf{i}Y_{4l+4})(-Y_{4l+3}-\mathbf{i}Y_{4l+4})+
(Y_{4l+1}+\mathbf{i}Y_{4l+2})(Y_{4l+1}-\mathbf{i}Y_{4l+2})\\&
=\sum_{k=1}^4Y_{4l+k}^2+\mathbf{i}[Y_{4l+3},Y_{4l+4}]
-\mathbf{i}[Y_{4l+1},Y_{4l+2}]=\sum_{k=1}^4Y_{4l+k}^2,
\end{aligned}\end{equation*}
by (\ref{com}), whose summation over $l$ gives us (\ref{3.9}) for $A'=B'=0'.$ Similarly we have
\begin{equation*}\begin{aligned}
\overline{Z_{2l}^{0'}}Z_{2l}^{1'}+
\overline{Z_{2l+1}^{0'}}Z_{2l+1}^{1'}&=
(-Y_{4l+3}+\mathbf{i}Y_{4l+4})(-Y_{4l+1}-\mathbf{i}Y_{4l+2})+
(Y_{4l+1}+\mathbf{i}Y_{4l+2})(-Y_{4l+3}+\mathbf{i}Y_{4l+4})\\&
=-[Y_{4l+1},Y_{4l+3}]-[Y_{4l+2},Y_{4l+4}]
+\mathbf{i}[Y_{4l+1},Y_{4l+4}]-\mathbf{i}[Y_{4l+2},Y_{4l+3}]=0,
\end{aligned}\end{equation*}
by (\ref{com}),  whose summation over $l$ gives us (\ref{3.9})  for $A'=0,B'=1'.$
Similarly,  (\ref{3.9}) holds for  $A'=1,B'=0'$ and $A'=B'=1'$ by
\begin{equation*}\begin{aligned}
\overline{Z_{2l}^{1'}}Z_{2l}^{0'}+
\overline{Z_{2l+1}^{1'}}Z_{2l+1}^{0'}
&=[Y_{4l+1},Y_{4l+3}]+[Y_{4l+2},Y_{4l+4}]
+\mathbf{i}[Y_{4l+1},Y_{4l+4}]-\mathbf{i}[Y_{4l+2},Y_{4l+3}]=
0,\\
\overline{Z_{2l}^{1'}}Z_{2l}^{1'}+
\overline{Z_{2l+1}^{1'}}Z_{2l+1}^{1'}
&=\sum_{k=1}^4Y_{4l+k}^2+\mathbf{i}[Y_{4l+1},Y_{4l+2}]
-\mathbf{i}[Y_{4l+3},Y_{4l+4}]=\sum_{k=1}^4Y_{4l+k}^2.
\end{aligned}\end{equation*}
Then (\ref{3.9}) follows.
\end{proof}
\section{The $L^2$ estimate}
We begin with the following Poincar\'e-type inequality, which  was proved for general vector fields satisfying H\"ormander's condition (cf.  \cite[Theorem 2.1]{Jerison}). So it holds over $\mathscr{H}.$
\begin{prop}
{\rm(Poincar\'e-type inequality)}
For each $f$ with $\sum_{a=1}^{4n}|Y_af|^2\in L^1(\mathscr{H}),$ we have
\begin{align}\label{poincare}
\int_{B_r}|f-f_{B_r}|^2 {\rm d}V\leq Cr^2\int_{B_r}\sum_{a=1}^{4n}|Y_af|^2 {\rm d}V,
\end{align}
where $B_r$ is a ball of radius $r$ and  $f_{B_r}={\int_{B_r}f{\rm d}V}/{\int_{B_r}{\rm d}V}.$
\end{prop}

We say $f\in L^2(\mathscr{H}/\mathscr{H}_{\mathbb{Z}},\mathscr{V}_1)$ satisfies  \emph{$f\perp{\,constant}$ vectors} if $\langle f,C\rangle=0$ for any constant vector $C\in \mathscr{V}_1.$
\begin{lem}\label{l3.3}
There exists   some $c>0$ such that
\begin{align*}
\left\langle \Delta_b f,f\right\rangle\geq c\|f\|^2_{L^2\left(\mathscr{H}/\mathscr{H}_{\mathbb{Z}}\right)},
\end{align*}
for  $f\in C^2\left(\mathscr{H}/\mathscr{H}_{\mathbb{Z}},\mathscr{V}_1\right)$ and $f\perp\,$constant vectors.
\end{lem}
\begin{proof}
As $\bigcup\limits_{(n,m)\in\mathscr{H}_{\mathbb{Z}}}
\tau_{(n,m)}\mathscr{F}=\mathscr{H}$ by Proposition \ref{p2.1}, we can choose some $r>0$ and a finite number of elements  $(n_i,m_i)\in \mathscr{H}_\mathbb{Z},i=1,\cdots,N,$ such that
\begin{equation*}
   \mathscr{F}\subset B_r\subset\bigcup\limits_{i=1}^N\tau_{(n_i,m_i)}\mathscr{F}.
\end{equation*}
  Recall that if we  identify $f\in C^2(\mathscr{H}/\mathscr{H}_{\mathbb{Z}},\mathscr{V}_1)$ with a periodic function on $\mathscr{H},$ so is $Y_af.$
Then the Poincar\'e-type inequality (\ref{poincare}) implies that
\begin{equation*}
\begin{aligned}
N\sum_{a=1}^{4n}\|Y_a f\|^2_{L^2\left(\mathscr{F}\right)}\geq  \sum_{a=1}^{4n}\|Y_a {f}\|^2_{L^2\left(B_r\right)}\geq \frac{1}{Cr^2}\int_{B_r}|{f}-f_{B_r}|^2\hbox{d}V\geq \frac{1}{Cr^2}
\|f-f_{B_r}\|^2_{L^2\left(\mathscr{F}\right)}.
\end{aligned}
\end{equation*}
Since $f\bot$ constant vectors, we have
$$\|f-f_{B_r}\|^2_{L^2(\mathscr{H}/\mathscr{H}_{\mathbb{Z}})}
=\|f\|^2_{L^2(\mathscr{H}/\mathscr{H}_{\mathbb{Z}})}
+\|f_{B_r}\|^2_{L^2(\mathscr{H}/\mathscr{H}_{\mathbb{Z}})}\geq
\|f\|^2_{L^2(\mathscr{H}/\mathscr{H}_{\mathbb{Z}})}.$$
Thus we find that
\begin{equation*}\begin{aligned}
\left\langle \Delta_b f,f\right\rangle\geq c
\|f-f_{B_r}\|^2_{L^2(\mathscr{H}/\mathscr{H}_{\mathbb{Z}})}\geq
c\|f\|^2_{L^2(\mathscr{H}/\mathscr{H}_{\mathbb{Z}})},
\end{aligned}
\end{equation*}
for  constant $c=\frac{1}{NCr^2}.$
\end{proof}
\begin{lem}\label{l3.5}{\rm(cf.  \cite[Lemma 2.1]{wang23})}
For any $h,H\in \mathbb{C}^{2n}\otimes\mathbb{C}^{2n},$ we have
\begin{align*}
\sum_{A,B}h_{BA}\overline{H_{AB}}=\sum_{A,B}h_{AB}\overline{H_{AB}}-
2\sum_{A,B}h_{[AB]}\overline{H_{[AB]}}.
\end{align*}
\end{lem}
We have the following $L^2$ estimate.
\begin{thm}
For $n>3,k\geq2,$ there exists
some $c_{n,k}>0$ such that
\begin{align}\label{3.0}
\|\mathscr{D}_{0}^*f\|^2+\|\mathscr{D}_{1}f\|^2\geq c_{n,k}\|f\|^2,
\end{align}
for $f\in Dom(\mathscr{D}_1)\cap Dom(\mathscr{D}_0^*)$ and  $f\perp\,${constant} vectors over $\mathscr{H}/\mathscr{H}_{\mathbb{Z}}$.
\end{thm}
\begin{proof}We use the $L^2$  method for the    $k$-Cauchy-Fueter operator on $\mathbb{H}^n$ in \cite{wang23}. Since $C^2$ functions are dense in $  Dom(\mathscr{D}_1)\cap Dom(\mathscr{D}_0^*)$ for the compact manifold $\mathscr{H}/\mathscr{H}_{\mathbb{Z}}$, it is sufficient to prove (\ref{3.0}) for $f\in C^2(\mathscr{H}/\mathscr{H}_{\mathbb{Z}},\odot^{k-1}\mathbb{C}^2\otimes
\mathbb{C}^{2n}).$
We have
\begin{equation}\begin{aligned}\label{3.10}
k\langle\mathscr{D}_{0}^*f,\mathscr{D}_{0}^*f\rangle=&
k\langle\mathscr{D}_{0}\mathscr{D}_{0}^*f,f\rangle=k\sum_{B,A_2',\cdots,A_k'}
\left(\sum_{A_1'}Z_B^{A_1'}\sum_{A}\delta^A_{(A_1'}f_{A_2'\cdots A_k')A},
f_{A_2'\cdots A_k'B}\right)\\
=&\quad\sum_{A,B,A_1',\cdots,A_k'}\left(Z_B^{A_1'}
\delta_{A_1'}^Af_{A_2' \cdots A_k'A},
f_{A_2' \cdots A_k'B}\right)\\&+
\sum_{A,B,A_1',\cdots,A_k'}\sum_{s=2}^{k}\left(Z_B^{A_1'}
\delta_{A_s'}^Af_{A_1'\cdots\widehat{A_s'}\cdots A_k'A},f_{A_2' \cdots A_k'B}\right)
=:\Sigma_0+\Sigma_1,
\end{aligned}\end{equation}
by using (\ref{2.35}) to expand the symmetrisation. Note that
\begin{equation}
\begin{aligned}\label{3.11}
\Sigma_0=\sum_{A_1',\cdots,A_k'}\left(\sum_A\delta_{A_1'}^A
f_{A_2' \cdots A_k'A},
\sum_B\delta_{A_1'}^Bf_{A_2' \cdots A_k'B}\right)=\sum_{A_1',\cdots,A_k'}
\left\|\sum_A\delta_{A_1'}^Af_{A_2' \cdots A_k'A}\right\|^2\geq0,
\end{aligned}
\end{equation}
and
\begin{equation}
\begin{aligned}\label{3.12}
\Sigma_1=&\quad\sum_{s=2}^{k}\sum_{A,B,A_1',\cdots,A_k'}
\left(\delta_{A_s'}^AZ_B^{A_1'}
f_{A_1'\cdots \widehat{A_s'}\cdots A_k'A},f_{A_2' \cdots A_k'B}\right)\\
&+\sum_{s=2}^{k}\sum_{A,B,A_1',\cdots,A_k'}
\left(\left[Z_B^{A_1'},\delta_{A_s'}^A
\right]
f_{A_1'\cdots \widehat{A_s'}\cdots A_k'A},f_{A_2' \cdots A_k'B}\right)=:
\Sigma_{11}+\mathscr C
\end{aligned}
\end{equation}
by using commutators.
For the first sum, we have
\begin{equation}\begin{aligned}\label{2.27}
\Sigma_{11}=&\sum_{s=2}^{k}\sum_{A,B,A_1',\cdots,A_k'}
\left(Z_B^{A_1'}f_{A_1'\cdots \widehat{A_s'}\cdots A_k'A},Z_A^{A_s'}f_{A_2' \cdots A_k'B}
\right)\\=&\sum_{s=2}^{k}\sum_{A,B}\sum_{\widehat{A_1'},\cdots,
\widehat{A_s'},\cdots,A_k'}
\left(\sum_{A_1'}Z_B^{A_1'}f_{A_1'\cdots \widehat{A_s'}\cdots A_k'A},\sum_{A_s'}Z_A^{A_s'}f_{A_s'A_2'\cdots \widehat{A_s'}\cdots A_k'B}
\right)\\=&(k-1)\sum_{B_3',\cdots,B_k'=0',1'}\sum_{A,B}
\left(\sum_{A'}Z_B^{A'}f_{A A'B_3'\cdots B_k'},\sum_{A'}Z_A^{A'}f_{B A'B_3'\cdots B_k'}\right)\end{aligned}\end{equation}
by relabelling indices and  $f$ symmetric in the primed indices. Then by applying Lemma \ref{l3.5} with $h_{BA}=\sum_{A'}Z_B^{A'}f_{AA'B_3'\cdots B_k'}$ and $H_{A B}=\sum_{A'}Z_A^{A'}f_{BA'B_3'\cdots B_k'}$ for fixed $B_3',\cdots,B_k',$ we get
\begin{equation}\begin{aligned}\label{2.28}
\Sigma_{11}=&(k-1)\sum_{B_3',\cdots,B_k'}\sum_{A,B}\left(
\left\|\sum_{A'}Z_{A}^{A'}f_{B A'B_3'\cdots B_k'}\right\|^2-2\left\|
\sum_{A'}Z_{[A}^{A'}f_{B]A'B_3'\cdots B_k'}\right\|^2\right)\\
=&(k-1)\sum_{B_3',\cdots,B_k'}\sum_{A,B}
\left\|\sum_{A'}Z_{A}^{A'}f_{B A'B_3'\cdots B_k'}\right\|^2-\frac{k-1}{2}\|\mathscr{D}_{1}f\|^2,
\end{aligned}\end{equation}
where
\begin{equation}\begin{aligned}\label{3.14}
\sum_A&\left\|\sum_{A'}Z_{A}^{A'}f_{B A'B_3'\cdots B_k'}\right\|^2=\sum_{A,A',B'}\left(Z_{A}^{A'}f_{B A'B_3'\cdots B_k'},
Z_{A}^{B'}f_{B B'B_3'\cdots B_k'}\right)\\
=&\sum_{A',B'}\left(-\sum_{A}\overline{Z_A^{B'}}
Z_A^{A'}f_{B A'B_3'\cdots B_k'},
f_{B B'B_3'\cdots B_k'}\right)=\sum_{B'}\left(\Delta_b f_{BB'B_3'\cdots B_k'},f_{BB'B_3'\cdots B_k'}\right)
\end{aligned}\end{equation}
by Lemma \ref{l3.2}. Thus  by substituting (\ref{3.11})-(\ref{3.12}) and (\ref{2.28})-(\ref{3.14}) to (\ref{3.10}), we get
\begin{align}\label{3.23}
k\left\|\mathscr{D}^*_0f\right\|^2+
\frac{k-1}{2}\left\|\mathscr{D}_1f\right\|^2\geq(k-1)\langle\Delta_b
f,f\rangle+\mathscr C.
\end{align}
To control the commutator term  $\mathscr C$ in (\ref{3.12}), note that
$$\overline{Z_{2l}^{0'}}=Z_{2l+1}^{1'},\quad \overline{Z_{2l}^{1'}}=-Z_{2l+1}^{0'}$$ by (\ref{ZAA}). Then it follows from Lemma \ref{l3.1} that (1)
\begin{align}\label{3.25}
\left[Z_A^{A'},\overline{Z_B^{B'}}\right]=0,\quad {\rm for}\ A'\neq B',
\end{align}
$A,B=0,\cdots,2n-1;$  (2) for $A'=B',$ we have
\begin{equation}\begin{aligned}\label{3.26}
\left[Z_{2l}^{0'},\overline{Z_{2l+1}^{0'}}\right]=&
\left[Z_{2l}^{1'},\overline{Z_{2l+1}^{1'}}\right]=
-8\left(\partial_{s_2}+\mathbf{i}\partial_{s_3}\right),\\
{\left[Z_{2l+1}^{0'},\overline{Z_{2l}^{0'}}\right]}=&
{\left[Z_{2l+1}^{1'},\overline{Z_{2l}^{1'}}\right]}
=8\left(\partial_{s_2}-\mathbf{i}\partial_{s_3}\right),\\
\left[Z_{2l}^{0'},\overline{Z_{2l}^{0'}}\right]=&
\left[Z_{2l}^{1'},\overline{Z_{2l}^{1'}}\right]=-
\left[Z_{2l+1}^{0'},\overline{Z_{2l+1}^{0'}}\right]=-
\left[Z_{2l+1}^{1'},\overline{Z_{2l+1}^{1'}}\right]
=8\mathbf{i}\partial_{s_1};
\end{aligned}\end{equation}
 (3) if  $\{A,B\}\neq\{2l,2l+1\}$ for any  $l,$ then $\left[Z_A^{A'},\overline{Z_B^{B'}}\right]=0$ for any $A',B'.$ Thus we have
\begin{equation*}\begin{aligned}\label{b}
\mathscr C&=
\sum_{s=2}^k\sum_{A,B,A_1',\cdots,A_k'}\left(\left[Z_B^{A_1'},
-\overline{Z^{A_s'}_A}\right]f_{A'_1\cdots \widehat{A_s'}\cdots A_k'A},f_{A_2'\cdots A_k'B}\right)\\&=-(k-1)
\sum_{A,B,B',B_3',\cdots,B_k'}\left(\left[Z_B^{B'},
\overline{Z^{B'}_A}\right]f_{B'B_3'\cdots B_k'A},f_{B'B_3'\cdots B_k'B}\right),\\&=-(k-1)\sum_{B',B_3',\cdots,B_k'}\sum_{l=0}^{n-1}
\left\{\left(\left[Z_{2l}^{B'},
\overline{Z_{2l}^{B'}}\right]f_{B'B_3'\cdots B_k'(2l)}
,f_{B'B_3'\cdots B_k'(2l)}\right)\right.\\&\hskip 42mm+\left(\left[Z_{2l}^{B'},
\overline{Z_{2l+1}^{B'}}\right]f_{B'B_3'\cdots B_k'(2l+1)}
,f_{B'B_3'\cdots B_k'(2l)}\right)\\ &\hskip 42mm+\left(\left[Z_{2l+1}^{B'},
\overline{Z_{2l}^{B'}}\right]f_{B'B_3'\cdots B_k'(2l)}
,f_{B'B_3'\cdots B_k'(2l+1)}\right)\\ &\hskip 42mm+\left.\left(\left[Z_{2l+1}^{B'},
\overline{Z_{2l+1}^{B'}}\right]f_{B'B_3'\cdots B_k'(2l+1)}
,f_{B'B_3'\cdots B_k'(2l+1)}\right)\right\}
\end{aligned}
\end{equation*}
by using (1) and (3) above, relabelling indices and $f$   symmetric in the primed indices. Apply (\ref{3.26}) to $\mathscr C$ above to get
\begin{equation}\begin{aligned}\label{b'}
\mathscr C
=-8(k-1)
\sum_{B',B_3',\cdots,B_k'}
&\left\{\sum_{A=0}^{2n-1}(-1)^{A}\left(\mathbf{i}\partial_{s_1}
f_{B'B_3'\cdots B_k'A},f_{B'B_3' \cdots B_k'A}\right)\right.\\&+\sum_{l=0}^{n-1}
\left(-( \partial_{s_2}+\mathbf{i}\partial_{s_3})f_{B'B_3'\cdots B_k'(2l+1)},
f_{B'B_3' \cdots B_k'(2l)}\right)\\ &\left.+\sum_{l=0}^{n-1}
\left((\partial_{s_2}-\mathbf{i}\partial_{s_3})f_{B'B_3' \cdots B_k'(2l)},
f_{B'B_3' \cdots B_k'(2l+1)}\right)\right\}.
\end{aligned}\end{equation}
For any $u,v\in C^1({\mathscr{H}}/\mathscr{H}_\mathbb{Z},\mathbb{C}),$ we have
\begin{equation*}\begin{aligned}
8\left(\partial_{s_1}u,v\right)=-\frac{1}{n}\sum_{l=0}^{n-1}\left(
 [Y_{4l+1},Y_{4l+2}]u+[Y_{4l+3},Y_{4l+4}]
u,v\right),
\end{aligned}\end{equation*}
by (\ref{com}). As
\begin{equation*}\begin{aligned}
\left|\left([Y_{a},Y_{b}]
u,v\right)\right|=\left|\left(Y_{b}
u,-{Y_{a}}v\right)+\left(Y_{a}
u,{Y_{b}}v\right)\right| \leq\frac{1}{2}\left( \|Y_{a}u\|^2+\|Y_{b}u\|^2+\|Y_{a}v\|^2+\|Y_{b}v\|^2\right),
\end{aligned}\end{equation*}
for $a,b=1,\cdots,4n,$ we get
\begin{equation}\begin{aligned}\label{3.28}
\left|8\left(\partial_{s_1}u,v\right)\right|\leq
\frac{1}{2n}\sum_{a=1}^{4n}\left(\|Y_{a}u\|^2+\|Y_{a}v\|^2\right).
\end{aligned}\end{equation}
Similarly, we have
\begin{equation}\begin{aligned}\label{3.29}
\left|\left(8(\partial_{s_2}\pm\mathbf{i}\partial_{s_3})
u,v\right)\right|&\leq\frac{1}{n}\sum_{a=1}^{4n}\left(\|Y_{a}u\|^2+
\|Y_{a}v\|^2\right).
\end{aligned}\end{equation}
Then apply (\ref{3.28})-(\ref{3.29}) to the right hand side of  (\ref{b'}) to get
\begin{align}\label{3.20}
|\mathscr C|\leq(k-1)\frac{3}{n}\sum_{A,B',B_3',\cdots,B_k'}
\sum_{a=1}^{4n}\left\|Y_{a}f_{B'B_3'\cdots B_k'A}\right\|^2=
\frac{3(k-1)}{n}\left\langle\Delta_bf,f\right\rangle.
\end{align}
So it follows from estimate (\ref{3.23}) that
\begin{equation}\begin{aligned}\label{3.22}
k\left\|\mathscr{D}^*_0f\right\|^2+
\frac{k-1}{2}\left\|\mathscr{D}_1f\right\|^2\geq
(k-1)\left(1-\frac{3}{n}
\right)\left\langle\Delta_bf,f\right\rangle.
\end{aligned}\end{equation}
Now by applying  Lemma \ref{l3.3}  we get (\ref{3.0}).
\end{proof}

\section{Hartogs' phenomenon}
\subsection{The  nonhomogeneous tangential $k$-Cauchy-Fueter equation over $\mathscr{H}/\mathscr{H}_{\mathbb{Z}}$}
Consider the Hilbert subspace $\mathcal{L}$ consisting of $f\in L^2\left(\mathscr{H}/\mathscr{H}_{\mathbb{Z}},\mathscr{V}_1\right)$ and $f\perp$ {constant} vectors. The domain of $\Box_1$ over $\mathcal{L}$ is
$${\rm Dom}(\Box_1):=\left\{f\in \mathcal{L}:f\in {\rm Dom}(\mathscr{D}_0^*)\cap {\rm Dom}(\mathscr{D}_1),\mathscr{D}_0^*f\in {\rm Dom}(\mathscr{D}_0),\mathscr{D}_1f\in {\rm Dom}(\mathscr{D}_1^*)\right\}.$$
\begin{prop}\label{p3.1}
The associated Hodge-Laplacian $\Box_1$  is densely-defined, closed, self-adjoint and nonnegative operator on  $\mathcal{L}.$
\end{prop}
The proof is exactly the same as  that of Proposition 3.1 in \cite{wang23} since $\mathcal{L}\oplus \{const.\}=L^2(\mathscr{H}/\mathscr{H}_\mathbb{Z},
\mathscr{V}_1)$,  and the action of $\Box_1$ on $\{const.\}$ is trivial.  We omit the detail.
Now we can find solution to (\ref{equ})-(\ref{equ1}), whose proof is similar to that of  Theorem 1.2 in \cite{wang23} for the $k$-Cauchy-Fueter operator  on $\mathbb{H}^n.$

\begin{thm}\label{t1} Suppose that ${\rm dim}\ \mathscr{H}\geq19 $ and $k=2,3,\ldots$.
If $f\in {\rm Dom}(\mathscr{D}_1)$ is $\mathscr{D}_1$-closed and $f\perp$ {constant} vectors, then there exist $u\in L^2\left(\mathscr{H}/\mathscr{H}_{\mathbb{Z}},\mathscr{V}_0\right)$ such that
\begin{align*}
\mathscr{D}_0u=f.
\end{align*}
\end{thm}
\begin{proof}
The $L^2$ estimate (\ref{3.0}) implies
\begin{align*}
c_{n,k}\|g\|^2\leq\|\mathscr{D}_{0}^*g\|^2+\|\mathscr{D}_{1}g\|^2
=\langle\Box_1g,g
\rangle\leq\|\Box_1g\|\|g\|,
\end{align*}
for  $g\in{\rm Dom}(\Box_1),$ i.e.
\begin{align}\label{3.11'}
c_{n,k}\|g\|\leq\|\Box_1g\|.
\end{align}
Thus $\Box_1:{\rm Dom}(\Box_1)\rightarrow \mathcal{L}$ is injective. This together with the  self-adjointness of $\Box_1$ by Proposition \ref{p3.1} implies the density of the range. For fixed  $f\in\mathcal{L},$ the complex anti-linear functional
\begin{align*}
l_f: \Box_1g\longrightarrow \langle f,g\rangle
\end{align*}
is then well-defined on a dense subspace of $\mathcal{L}.$ It is finite since
\begin{align*}
|l_f(\Box_{1}g)|=|\langle f,g\rangle|\leq\|f\|\|g\|\leq\frac{1}{c_{n,k}}\|f\|\|\Box_1g\|
\end{align*}
for any $g\in{\rm Dom}(\Box_1),$ by (\ref{3.11'}). So $l_f$ can be uniquely extended to a continuous anti-linear functional on $\mathcal{L}.$ By the Riesz representation theorem, there exists a unique element $h\in \mathcal{L}$ such that $l_f(F)=\langle h,F\rangle$ for any $F\in \mathcal{L},$ and $\|h\|=\|l_f\|\leq\frac{1}{c_{n,k}}\|f\|.$ Then, we have
$$\langle h,\Box_1g\rangle=\langle f,g\rangle$$
for any $g\in{\rm Dom}(\Box_1).$ This implies that $h\in{\rm Dom}(\Box_1^*)$ and $\Box_1^*h=f,$ and so $h\in{\rm Dom}(\Box_1)$ and $\Box_1h=f$ by self-adjointness of $\Box_1.$ We write $h=Nf.$ Then $\|Nf\|\leq\frac{1}{c_{n,k}}\|f\|.$

Since $Nf\in{\rm Dom}(\Box_1),$ we have $\mathscr{D}_0^*Nf\in{\rm Dom}(\mathscr{D}_0),$ $\mathscr{D}_1Nf\in{\rm Dom}(\mathscr{D}_1^*),$ and
\begin{align}\label{3.12'}
\mathscr{D}_0\mathscr{D}_0^*Nf=f-\mathscr{D}_1^*\mathscr{D}_1Nf
\end{align}
by $\Box_1Nf=f.$ Because $f$ and $\mathscr{D}_0F$ for any $F\in{\rm Dom}(\mathscr{D}_0)$ are both $\mathscr{D}_1$-closed, the above identity implies $\mathscr{D}_1^*\mathscr{D}_1Nf\in{\rm Dom}(\mathscr{D}_1)$ and so $\mathscr{D}_1\mathscr{D}_1^*\mathscr{D}_1Nf=0.$ Then
$$0=\langle\mathscr{D}_1\mathscr{D}_1^*\mathscr{D}_1Nf,\mathscr{D}_1Nf
\rangle=\|
\mathscr{D}_1^*\mathscr{D}_1Nf\|^2,$$
i.e. $\mathscr{D}_1^*\mathscr{D}_1Nf=0.$ Hence $\mathscr{D}_0\mathscr{D}_0^*Nf=f$ by (\ref{3.12'}).
\end{proof}

\subsection{Proof of Hartogs' phenomenon}
We need the analytic hypoellipticity of $\Delta_b$.
Let $G$ be a nilpotent Lie group of step $2,$ and its Lie algebra $\mathfrak{g}$  has decomposition:
$\mathfrak{g}=\mathfrak{g}_1\oplus\mathfrak{g}_2$
satisfying $[\mathfrak{g}_1,\mathfrak{g}_1]\subset\mathfrak{g}_2,\
[\mathfrak{g},\mathfrak{g}_2]=0.$ Consider the  \emph{condition (H)}: For any $\lambda\in\mathfrak{g}_2^*\setminus\{0\},$ the anti-symmetric bilinear form
 $$B_\lambda(Y,Y')=\langle\lambda,[Y,Y']\rangle,$$
for  $Y,Y'\in \mathfrak{g}_1$ is nondegenerate. M\'etivier proved the following theorem for analytic hypoellipticity.
\begin{thm}\label{elliptic}{\rm (\cite[Theorem 0]{Metivier2})}
Let $P$ be a homogeneous  left invariant differential operator on a nilpotent Lie group $G$ satisfies condition (H). Then the following are equivalent:\\ $(i)$ $P$ is analytic hypoelliptic;\\
$(ii)$ $P$ is $C^\infty$ hypoelliptic.
 \end{thm}
\begin{cor}\label{anal}
$\Delta_b$ is analytic hypoelliptic on a domain $\Omega\subset\mathscr{H},$ i.e.  for any distribution   $u\in S'(\Omega)$   such that $\Delta_bu$ is analytic, $u$ must be also analytic.
\end{cor}
\begin{proof} It follows from the well known subellipticity of $\Delta_b $ that $u$ is locally $C^{k+1}$  if  $\Delta_bu$ is  locally $C^{k }$. So $\Delta_b $ is $C^\infty$ hypoelliptic. To obtain the analytic hypoellipticity of $\Delta_b $ by applying Theorem \ref{elliptic}, it is sufficient to
  check the condition $(H)$ for the right quaternionic Heisenberg group $\mathscr{H}$. In this case $\mathfrak{g}_1={\rm span}\{Y_1,\cdots,Y_{4n}\},$ $\mathfrak{g}_2={\rm span}\left\{ \partial_{s_1}, \partial_{s_2}, \partial_{s_3}\right\},$ where $Y_1,\cdots,Y_{4n}$ is the left invariant vector fields  in (\ref{2.43}). Let $\lambda\in \mathfrak{g}_2^*\setminus\{0\}.$ For  $Y_{4l+j},Y_{4l+j'}\in \mathfrak{g}_1,$ we have
\begin{align*}
B_\lambda(Y_{4l+j},Y_{4l'+j'})=\langle\lambda,[Y_{4l+j},Y_{4l'+j'}]\rangle=4\delta_{ll'}
\sum_{\beta=1}^{3}B_{jj'}^{\beta}\lambda(\partial_{s_{\beta}})=4\delta_{ll'}
\sum_{\beta=1}^{3}B_{jj'}^{\beta}\lambda_{\beta},
\end{align*}
by (\ref{2.16'}), if we write $\lambda(\partial_{s_{\beta}})=\lambda_{\beta} $. Then the matrix associated to $B_\lambda$ is
\begin{align}\label{Blambda} \sum_{\beta=1}^{3}\left(
\begin{array}{ccc}\lambda_{\beta}B^{\beta}&&\\&\ddots&\\
&&\lambda_{\beta}B^{\beta}\end{array}\right),\quad {\rm where}\ \sum_{\beta=1}^{3}\lambda_{\beta}B^{\beta}=\left(
\begin{array}{cccc}0&-\lambda_1&-\lambda_2&-\lambda_3\\ \lambda_1&0&-\lambda_3
&\lambda_2\\\lambda_2&\lambda_3&0&-\lambda_1\\ \lambda_3&-\lambda_2
&\lambda_1&0\end{array}\right),\end{align}
 whose determinant is $ \left(\lambda_1^2+\lambda_2^2+\lambda_3^2\right)^{2n}$ by directly calculation. So $B_\lambda$ is nondegenerate for $\lambda\in \mathfrak{g}_2^*\setminus\{0\},$ i.e. $\mathscr{H}$ satisfies condition $(H).$
\end{proof}
Liouville-type theorems holds for  SubLaplacian $\Delta_b$ on the right quaternionic Heisenberg group by the following general theorem of Geller.
\begin{thm}\label{liou}{\rm(\cite[Theorem 2]{Daryl})}
Let $\mathscr{L}$ be a homogeneous hypoelliptic left invariant differential operator on a homogeneous group $G.$ Suppose $u\in S'(G)$ and $\mathscr{L}u=0.$ Then $u$ is a polynomial.
\end{thm}
\begin{thm}\label{kcfe}
Let $\widetilde{\Omega}$ be an open set in $\mathscr{F}$  such that $\widetilde{{\Omega}}\Subset\mathring{\mathscr{F}}$ and $\mathscr{F}\setminus\widetilde{\Omega}$ is connected.  If $f\in C^1(\mathscr{H}/\mathscr{H}_{\mathbb{Z}},\mathscr{V}_1)$ with ${\rm supp}f\subset\widetilde{\Omega}$  is $\mathscr{D}_1$-closed and $f\perp$ {constant} vectors, then there exist $u\in C^2\left(\mathscr{H}/\mathscr{H}_{\mathbb{Z}},\mathscr{V}_0\right)$ such that
\begin{equation} \label{d0u}
\mathscr{D}_0u=f,
\end{equation}
with ${\rm supp}\,u\subset\widetilde{\Omega}.$
\end{thm}
\begin{proof}
By Theorem \ref{t1}, we can find a solution $u\in L^2\left(\mathscr{H}/\mathscr{H}_{\mathbb{Z}},\mathscr{V}_0\right)$ to (\ref{d0u}).   For $c\in  \mathbb{{H}}$, denote the subgroup
$$\mathscr{H}_c':=\{(q',c,s)\in\mathscr{H}:q'\in
\mathbb{H}^{n-1},s\in\mathbb{R}^3\}.$$ We see that $\mathscr{H}'_c\cap\Omega=\emptyset$ for $|c|$ small by $\widetilde{\Omega}\Subset\mathring{\mathscr{F}}.$

Since $\mathscr{D}_0u=0$ on $\left(\mathscr{H}/\mathscr{H}_{\mathbb{Z}}\right)\setminus\widetilde{\Omega},$  we have $\mathscr{D}_0^*\mathscr{D}_0u=0,$ and then by Proposition \ref{DD}, $\Delta_bu_{A'_1\cdots A'_{k }A}$ $=0$ on $\left(\mathscr{H}/\mathscr{H}_{\mathbb{Z}}\right)\setminus\widetilde{\Omega}$ in the  sense of distributions for any fixed $A'_1,\cdots ,A'_{k },A$. So it is real analytic on $\left(\mathscr{H}/\mathscr{H}_{\mathbb{Z}}\right)\setminus\widetilde{\Omega}$ by Corollary \ref{anal}. Moreover, $u$ is $C^2$ on $\mathscr{H}/\mathscr{H}_{\mathbb{Z}}$ by subellipticity of $\Delta_b $. In particular $u(q',c,s)$ is well-defined on $\mathscr{H}'_c/\mathscr{H}'_{\mathbb{Z}}$ as a real analytic function. So it can be extended to a periodic function over $\mathscr{H}'_c$ by (\ref{3.30}).
Now let $\mathscr{D}_0'$ be the tangential $k$-Cauchy-Fueter operator on $\mathscr{H}_c',$ i.e. $\mathscr{D}_0'u$ is a $\odot^{k-1}\mathbb{C}^2\otimes\mathbb{C}^{2n-2}$-valued function with
  \begin{equation*}
     \left(\mathscr{D}_0' u\right)_{AA_2'\cdots A_{k}'}= (\mathscr{D}_0 u )_{AA_2'\cdots A_{k}'} , \qquad A=0,1,\cdots,2n-3.
  \end{equation*}
   By applying Proposition \ref{DD} to $\mathscr{H}_c',$ we see that    $\Delta_b'u=0,$ where $\Delta_b'=-\sum_{a=0}^{4n-5}Y_a^2.$
Then apply  Liouville-type Theorem \ref{liou} to the group $\mathscr{H'}_c$ and $\Delta_b'$  to get
$$u(\cdot,c,\cdot)= {\rm a\ polynomial}\ {\rm on}\ \mathscr{H}'_c,$$ which must be a constant by periodicity. Thus $u$ only depends on the variable $q_n.$

Similarly, we can prove $u$ is a constant on the subgroup
  \begin{equation*}
    \mathscr{H}_0'':=\{(0,q_n,s)\in\mathscr{H};q_n\in
\mathbb{H},s\in\mathbb{R}^3\}. \end{equation*}
Now if replacing $u$ by  $u-$  const., we see that $u$ vanishes in a neighborhood of  $\mathscr{H}_0''.$   Consequently by the identity theorem for real analytic functions  it  vanishes on the connected component   $\mathscr{F}\setminus\widetilde{{\Omega}}.$ Thus ${\rm supp}\,u\subset\widetilde{\Omega}.$
\end{proof}
The solution with supp$\,u\subset\widetilde{\Omega}$ above plays the role of compactly supported solution to $\overline{\partial}$ equation or the tangential  $k$-Cauchy-Fueter equations (cf. e.g. \cite{Homander,wang15}). It leads to  Hartogs'
extension phenomenon as  follows.
\vskip 5mm
{\it The proof of Theorem \ref{hartogs}.}
Without loss of generality, we can assume $\Omega\Subset\mathring{\mathscr{F}}$ by dilating if necessary.
Let $\chi\in C_0^{\infty}(\Omega)$ be equal to $1$ in a neighborhood of $K$ such that $\mathscr{F}\setminus{\rm supp}\,\chi$ is connected. Set
 \begin{equation*}
    \widetilde{u}(\xi):=\left\{\begin{array}{lll}&(1-\chi)u(\xi),\quad &\xi\in\Omega\setminus K\\&0,\quad &\xi\in K\end{array}
\right..
 \end{equation*}
Then $\widetilde{u}\in C^{\infty}(\Omega),$ and $\widetilde{u}|_{\Omega\setminus {\rm supp}\,{\chi}}=u|_{\Omega\setminus {\rm supp}\,\chi}.$ We have
 \begin{equation*}
    \mathscr{D}_0\tilde{u}=\mathscr{D}_0
((1-\chi) {u})=:f
 \end{equation*}
on $\mathscr{H},$ where  $f_{A_2'\cdots A_k'A}=-\sum_{A_1'}Z_A^{A_1'}\chi\cdot u_{A_1'\cdots A_k'}$ by $\mathscr{D}_0u=0$ on $\Omega\setminus K.$ Hence $f\in C_0^\infty(\mathscr{H},\mathscr{V}_1)$ vanishes in $K$ and outside $\Omega,$  satisfying $\mathscr{D}_1f=\mathscr{D}_1\mathscr{D}_0\tilde{u}=0$ by (\ref{d10}). We can extend $f$ to a periodic function and   view it as an element of $C^\infty(\mathscr{H}/\mathscr{H}_\mathbb{Z},\mathscr{V}_1).$

Denote
\begin{equation*}
   c:=\frac {\int_{\mathscr{H}/\mathscr{H}_\mathbb{Z}}f\hbox{d}V}{
\int_{\mathscr{H}/\mathscr{H}_\mathbb{Z}}\hbox{d}V}\in\mathscr{V}_1.
\end{equation*}
Then we have $(f-c) \perp\,$constant vectors.
It follows from Theorem \ref{kcfe} that there exists   a solution $\widetilde{U}\in C^2(\mathscr{H}/\mathscr{H}_\mathbb{Z},\mathscr{V}_0)$ to $\mathscr{D}_0\widetilde{U}=f-c,$ which vanishes outside $\widetilde{\Omega}:={\rm supp}\,\chi$. Then $\mathscr{D}_0(\widetilde{u}-\widetilde{U})=c$ on $\mathscr{H}/\mathscr{H}_\mathbb{Z}.$    So $c=\mathscr{D}_0\widetilde{u}|_{\Omega\setminus \widetilde{\Omega}}=\mathscr{D}_0 {u}|_{\Omega\setminus \widetilde{\Omega}}=0.$
Therefore $
U=\widetilde{u}-\widetilde{U}
$
is $k$-CF in $\Omega$ since $\mathscr{D}_0(\widetilde{u}-\widetilde{U})=0.$
Note that $\widetilde{U}\equiv0$ outside   $\widetilde{\Omega}$ and $ \mathscr{F} \setminus \widetilde{\Omega}$ is connected. So $U=u$ in $\Omega\setminus \widetilde{\Omega}.$ Then $U=u$ in $\Omega\setminus K$ by the identity theorem for real analytic functions. The theorem is proved.

\section{The restriction of the  $k$-Cauchy-Fueter operator to the hypersurface $\mathcal{S}$}
\subsection{The nilpotent Lie groups of step two associated to   quadratic hypersurfaces}Let $(x_1,\cdots,x_{4n},$ $t_1,t_2,t_3)$ be   coordinates of $\mathbb{R}^{4n+3}.$
Now consider general quadratic hypersurfaces $\widehat{\mathcal{S}}$ defined by
\begin{equation*}
    \rho={\rm Re}q_{n+1}-\phi(q'),\qquad {\rm where}\quad  \phi=\sum_{k=1}^{4n}\mathbb S_{jk}x_jx_k,
\end{equation*}
for some symmetric matrix $\mathbb S.$ Define the projection:
\begin{equation}\begin{aligned}\label{pi}
\pi:\qquad\qquad\widehat{\mathcal{S}}\qquad\qquad&\longrightarrow\mathbb{H}^n
\times{\rm Im}\,\mathbb{H}\simeq\mathbb{R}^{4n+3},\\(q_1,\cdots,q_n,
\phi(q')+\mathbf{t})&\longmapsto
(q_1,\cdots,q_n,\mathbf{t}),
\end{aligned}\end{equation}
where $\mathbf{t}=t_1\textbf{i}+t_2\textbf{j}+t_3\textbf{k},$ $q_{l+1}=x_{4l+1}+\textbf{i}x_{4l+2}+\textbf{j}x_{4l+3}+\textbf{k}x_{4l+4},$  $l=0,\cdots,n-1$ and $t_\beta=x_{4n+1+\beta}$ for $\beta=1,2,3.$ Let $\psi:\mathbb{H}^{n}\times{\rm Im}\,\mathbb{H}\longrightarrow\mathcal{S}\in\mathbb{H}^{n+1}$ be its inverse.     The Cauchy-Fueter operator is $$\overline{\partial}_{q_{l+1}}=\partial_{x_{4l+1}}+\textbf{i}
\partial_{x_{4l+2}}+\textbf{j}\partial_{x_{4l+3}}+\textbf{k}
\partial_{x_{4l+4}}.$$
Then $\overline{\partial}_{q_{l+1}}+\overline{\partial}_{q_{l+1}}\phi\cdot
\overline{\partial}_{q_{n+1}}$ is a vector field tangential to the hypersurface $\widehat{\mathcal{S}},$ since   $$\left(\overline{\partial}_{q_{l+1}}+\overline{\partial}_{q_{l+1}}\phi\cdot
\overline{\partial}_{q_{n+1}}\right)\rho=0,$$ $l=0,1,\cdots,n-1.$  This vector field is exactly the pushforward vector field $\psi_*\left(\overline{\partial}_{q_{l+1}}+\overline{\partial}_{q_{l+1}}\phi\cdot
\overline{\partial}_{\mathbf{t}}\right)$, where $\overline{\partial}_{\mathbf{t}}=\textbf{i}{\partial}_{t_{1}}+
\textbf{j}{\partial}_{t_{2}}+\textbf{k}
{\partial}_{t_{3}}.$
Because
$$\psi_*\partial_{t_{\beta}}=\partial_{x_{4n+1+\beta}},\quad \psi_*\partial_{x_{4l+j}}=\partial_{x_{4l+j}}+\partial_{x_{4l+j}}\phi\cdot
\partial_{x_{4n+1}},$$
for $\beta=1,2,3,j=1,\cdots,4,l=0,\cdots,n-1 $, and
\begin{equation*}\begin{aligned}
\psi_*\left(\overline{\partial}_{q_{l+1}}+\overline{\partial}_{q_{l+1}}\phi\cdot
\overline{\partial}_{\mathbf{t}}\right)=&\sum_{j=1}^4\textbf{i}_{j-1}\left(
{\partial}_{x_{4l+j}}+{\partial}_{x_{4l+j}}\phi\cdot
{\partial}_{x_{4n+1}}\right)\\&+
\overline{\partial}_{q_{l+1}}\phi\left(\textbf{i}{\partial}_{x_{4n+2}}+
\textbf{j}{\partial}_{x_{4n+3}}+\textbf{k}
{\partial}_{x_{4n+4}}\right)=
\overline{\partial}_{q_{l+1}}+\overline{\partial}_{q_{l+1}}\phi\cdot
\overline{\partial}_{q_{n+1}}
,
\end{aligned}\end{equation*}
Denote
\begin{equation}\label{eq:Xj}
   X_{4l+1}+\mathbf{i}X_{4l+2}+\mathbf{j}X_{4l+3}+\mathbf{k}X_{4l+4}:=\overline{\partial}_{q_{l+1}}+\overline{\partial}_{q_{l+1}}\phi\cdot
\overline{\partial}_{\mathbf{t}}.
\end{equation}
\begin{prop} \label{prop:X} We have
\begin{align*}
X_{b}=\partial_{x_{b}}+2\sum_{\beta=1}^3\sum_{a=1}^{4n}\left(\mathbb S\mathbb I^\beta\right)_{ab}x_a\partial_{t_\beta},
\end{align*}
where $\mathbb I^\beta$ is the $(4n)\times(4n)$ matrix ${\rm diag}\left(I^\beta,\cdots,I^\beta\right).$

\end{prop}
\begin{proof}
The proof is similar to that of Proposition 2.1 in \cite{wang2}. Consider right multiplication by  $\textbf{i}_\beta $. Noting  that
\begin{equation*}\begin{aligned}
(x_1+x_2\textbf{i}+x_3\textbf{j}+x_4\textbf{k})\textbf{i}
=-x_2+x_1\textbf{i}+x_4\textbf{j}-x_3\textbf{k},\\
(x_1+x_2\textbf{i}+x_3\textbf{j}+x_4\textbf{k})\textbf{j}
=-x_3-x_4\textbf{i}+x_1\textbf{j}+x_2\textbf{k},\\
(x_1+x_2\textbf{i}+x_3\textbf{j}+x_4\textbf{k})\textbf{k}
=-x_4+x_3\textbf{i}-x_2\textbf{j}+x_1\textbf{k},
\end{aligned}\end{equation*}
we can write
\begin{align}\label{2.8}
(x_1+x_2\textbf{i}+x_3\textbf{j}+x_4\textbf{k})\textbf{i}_\beta=-(I^\beta x)_1-(I^\beta x)_2\textbf{i}-(I^\beta x)_3\textbf{j}-(I^\beta x)_4\textbf{k}=-\sum_{j=1}^4(I^\beta x)_j\textbf{i}_{j-1},
\end{align}
where $I^\beta$'s are given by (\ref{I}).
$B^\beta$ in (\ref{2.14}) is the matrix  associated to left multiplication by  $\mathbf{i}_\beta$ (\cite[pp. 1358]{wang2}). Then we have \begin{equation*}\begin{aligned}
\overline{\partial}_{q_{l+1}}\phi\cdot\partial_\mathbf t=&\left(\partial_{x_{4l+1}}\phi+\textbf{i}\partial_{x_{4l+2}}\phi+
\textbf{j}\partial_{x_{4l+3}}\phi+\textbf{k}\partial_{x_{4l+4}}\phi\right)
\left(
\textbf{i}\partial_{t_{1}}+
\textbf{j}\partial_{t_{2}}+\textbf{k}\partial_{t_{3}}\right)\\=&
-\sum_{\beta=1}^3
\sum_{j,k=1}^{4}I_{jk}^\beta\partial_{x_{4l+k}}\phi\textbf{i}_{j-1}
\partial_{t_{\beta}}.
\end{aligned}\end{equation*}
Substitute it into (\ref{eq:Xj}) to get
$$X_{4l+j}=\partial_{x_{4l+j}}+2\sum_{\beta=1}^3\sum_{k=1}^{4}\sum_{a=1}^{4n}I_{kj}^\beta \mathbb S_{a(4l+k)}x_a\partial_{t_\beta}=\partial_{x_{4l+j}}+2\sum_{\beta=1}^3\sum_{a=1}^{4n}
\left(\mathbb S\mathbb I^\beta\right)_{a(4l+j)}x_a\partial_{t_\beta},$$
by the anti-symmetry of $I^\beta.$
 \end{proof}
By Proposition \ref{prop:X} we get
\begin{equation*}
   [X_{a},X_{b}]=2\sum_{\beta=1}^{3}\left(\left(\mathbb{S}\mathbb I^\beta\right)_{ab}-\left(\mathbb{S}\mathbb I^\beta\right)_{ba}\right)\partial_{t_\beta}.
\end{equation*}
So ${\rm span}_\mathbb{C}\left\{X_1,\cdots,X_{4n},\partial_{t_1},
\partial_{t_2},\partial_{t_3}\right\}$ is a nilpotent Lie algebra with center ${\rm span}_\mathbb{C}\left\{\partial_{t_1},
\partial_{t_2},\partial_{t_3}\right\}.$ The corresponding nilpotent Lie groups of step two is the group associated to the  quadratic hypersurface $\widehat{\mathcal{S}}$.

So if we choose the matrix $\mathbb S$ so that
 \begin{align*}
 \mathbb{S}\mathbb I^\beta-\mathbb I^\beta\mathbb{S}=2\mathbb B^\beta,
 \end{align*}
 where $\mathbb B^\beta={\rm diag}\left(B^\beta,\cdots,B^\beta\right),$ then the Lie algebra spanned by $X_1,\cdots,X_{4n},\partial_{t_1},\partial_{t_1},\partial_{t_3}$ is  isomorphic to the Lie algebra of the right quaternionic   Heisenberg group $\mathscr H.$ It is sufficient to choose $\mathbb{S}={\rm diag}(S,\cdots,S)$ such that $SI^\beta-I^\beta S=2B^\beta,$ where $S$ is a symmetric $4\times4$ matrix. Namely \begin{align}\label{CS}
 C^\beta-\left(C^\beta\right)^t=2B^\beta,
 \end{align} for $C^\beta=SI^\beta.$
 Then $S={\rm diag}(-3,1,1,1,),$ i.e.
\begin{equation*}
\begin{aligned}
C^{1}:=\left(\begin{array}{cccc} 0 & -3 & 0 &0\\ -1& 0& 0& 0\\ 0& 0&0& -1\\0 &0& 1 &0\end{array}\right),    C^{2}:=\left(\begin{array}{cccc} 0 & 0 &
-3 &0\\ 0& 0& 0& 1\\ -1& 0&0& 0\\0 &-1& 0 &0\end{array}\right),
C^{3}:=\left(\begin{array}{cccc} 0 & 0 & 0 &-3\\ 0& 0& -1& 0\\ 0& 1&0& 0\\-1 &0& 0 &0\end{array}\right).
\end{aligned}
\end{equation*}
satisfy (\ref{CS}). Thus $\widehat{\mathcal S}$ in this case is the hypersurface $\mathcal{S}$, and  the Lie group associate to $\mathcal S$ is the right quaternionic  Heisenberg group.
\subsection{The restriction of the $k$-Cauchy-Fueter operator.} $X_a$'s for $\mathcal{S}$ has the form
\begin{equation*}
   X_{4l+j}= \partial_{x_{4l+j}}+2\sum_{\beta=1}^3\sum_{k=1}^{4 }
C_{kj}x_k\partial_{t_\beta}.
\end{equation*}
Since $C^\beta$ is not anti-symmetric, they are different from the standard  left invariant vector fields (\ref{2.43}) on  $\mathscr{H} $.
 It is standard  that they can be transformed to the standard  left invariant vector fields (\ref{2.43}) on  $\mathscr{H} $ by a simple coordinate transformation  $\mathcal{F}:\mathscr{H}\rightarrow \mathbb{R}^{4n+3},(y,s)\mapsto(x,t)$ given by
\begin{align}\label{F}
x_{4l+j}=y_{4l+j},\quad t_\beta=s_\beta+\sum_{k,j=1}^4D^\beta_{kj}y_{4l+k}y_{4l+j},
\end{align}
(cf. \cite[(1.8)]{wang2}) with
$
D^\beta:=C^\beta+\left(C^\beta\right)^t
$
symmetric. It is direct to see that
\begin{equation*}
   \mathcal{F}_*\partial_{s_\beta}=\partial_{t_\beta}\qquad {\rm and} \qquad   \mathcal{F}_*Y_{4l+j}=X_{4l+j},
\end{equation*}
where $Y_{4l+j}$ is given by (\ref{2.43}).
 Then we find the relationship between complex horizontal vector fields  $Z_A^{A'}$'s  on $\mathscr {H} $ and $\nabla_A^{A'}$'s on $ \mathbb H^{n+1}$.
\begin{prop}\label{p2.3} Under the diffeomorphism  $\psi\circ\mathcal{F}:\mathscr{H}\rightarrow \mathcal{S}$, we have
\begin{align}\label{2.18}
\left(\psi\circ\mathcal{F}\right)_*Z_A^{A'}
=\nabla_A^{A'}+\sum_{\alpha=0,1}C_A^{\alpha}\nabla_{(2n+\alpha)}^{A'},\quad {\rm for}\ \left(C_A^{\alpha}\right):=\left(\begin{smallmatrix}\vdots\\
\tau \left(\overline{\partial}_{q_l}\phi\right) \\
\vdots\end{smallmatrix}\right),
\end{align}for fixed $A=0,1,\cdots,2n-1,A'=0',1' $,
  where $\tau$ is the embedding given by   (\ref{tau}).
\end{prop}
\begin{proof}
As $\tau $ is a representation,   we have
\begin{equation*}\begin{aligned}
\psi_*&\left(\begin{array}{ll}-X_{4l+3}-\mathbf{i}X_{4l+4} &-X_{4l+1}-\mathbf{i}X_{4l+2} \\\ \ X_{4l+1}-\mathbf{i}X_{4l+2} & -X_{4l+3}+\mathbf{i}X_{4l+4}\  \end{array}\right)=\psi_* \left(\begin{array}{rr}X_{4l+1}+\mathbf{i}X_{4l+2} & -X_{4l+3}-\mathbf{i}X_{4l+4} \\ X_{4l+3}- \mathbf{i}X_{4l+4}&   X_{4l+1}-\mathbf{i}X_{4l+2} \end{array}\right)\varepsilon\\&=
\tau \left(\psi_*(X_{4l+1}+\mathbf{i}X_{4l+2}+\mathbf{j}X_{4l+3}
+\mathbf{k}X_{4l+4})\right)\varepsilon\\&=
\tau \left(\overline{\partial}_{q_{l+1}}+\overline{\partial}_{q_{l+1}}\phi\cdot
\overline{\partial}_{q_{n+1}}\right)\varepsilon=\tau\left (\overline{\partial}_{q_{l+1}}\right)\varepsilon
+\tau \left(\overline{\partial}_{q_{l+1}}
\phi\right)\tau\left (\overline{\partial}_{q_{n+1}}\right)\varepsilon \\&=\left(\begin{array}{cc} \nabla_{(2l)}^{0'}&\nabla_{(2l)}^{1'}\\\nabla_{(2l+1)}^{0'}&\nabla_{(2l+1)}^{1'}
\end{array}\right)+\tau (\overline{\partial}_{q_l}
\phi)\left(
\begin{matrix}\nabla_{(2n)}^{0'}&\nabla_{(2n)}^{1'}\\
\nabla_{(2n+1)}^{0'}&\nabla_{(2n+1)}^{1'}\end{matrix}\right),
\end{aligned}\end{equation*}
where $\varepsilon=\left(
\begin{matrix}0&-1\\
1&0\end{matrix}\right)   $   in (\ref{epsilon}).
Then
(\ref{2.18}) follows.
\end{proof}
From this proposition we can derive the relationship between operators in $k$-Cauchy-Fueter complex on $\mathbb{H}^{n+1}$ and that in the tangential $k$-Cauchy-Fueter complex  on $\mathscr {H} $.
\begin{prop}\label{p1.1}
Suppose that $f$ is a $k$-regular function near $q_0\in\mathcal{S}.$ Then $\left(\psi\circ\mathcal{F}\right)^*f$ is  $k$-CF  on $\mathscr{H}$ near the point $\mathcal{F}^{-1}(\pi(q_0)) $.
\end{prop}
\begin{proof}As $f$ is a $k$-regular function near $q_0\in\mathcal{S}\subset \mathbb{H}^{n+1},$ we have $\sum_{B'=0',1'}\nabla_{A}^{B'}f_{B' A_2'\cdots A_k'}=0$ for any  fixed $A =0,1,\cdots,2n+1, $ $A_2',\cdots, A_k' =0',1'.$
Then   we  find that
\begin{equation*}\begin{aligned}
\left.\left(\mathscr{D}_{0}(\psi\circ\mathcal{F})^*f\right)_{A  A_2'\cdots A_k'}
\right|_{\mathcal{F}^{-1}(\pi(q_0))} =&\sum_{B'=0',1'}\left.Z_{A}^{B'}((\psi\circ\mathcal{F})^*f)_{B' A_2'\cdots A_k'}\right|_{\mathcal{F}^{-1}(\pi(q_0))}\\=&
\sum_{B'=0',1'}\left(\psi\circ\mathcal{F}\right)_*Z_{A}^{B'}f_{B'A_2'\cdots A_k'}(q_0)
\\=&\sum_{B'}\left(\nabla_{A}^{B'}+\sum_{\alpha=0,1}C_A^{\alpha}
\nabla_{(2n+\alpha)}^{B'}\right)f_{B' A_2'\cdots A_k'}(q_0)=0,
\end{aligned}\end{equation*}for any  fixed $A =0,1,\cdots,2n-1, $ $A_2',\cdots, A_k' =0',1' $, by Proposition \ref{p2.3}.
The proposition is proved.
\end{proof}

\section{Appendix}

In the case $n=2,k=2,$ We have isomorphisms
\begin{align}\label{id1}\odot^{2}\mathbb{C}^{2}\cong\mathbb{C}^{3},\quad \ \ \ \mathbb{C}^{2}\otimes\mathbb{C}^{4}\cong\mathbb{C}^{8},
\end{align}
by  identifying $f\in\odot^{2}\mathbb{C}^{2}$ and $F\in\mathbb{C}^{2}\otimes\mathbb{C}^{4}$ with
\begin{equation}\label{id2}
f:=\left(\begin{array}{c}f_{0'0'}\\f_{0'1'}\\
f_{1'1'}\end{array}\right),\quad F:=\left(\begin{array}{c}F_{0'0}\\\vdots\\F_{0'3}\\F_{1'0}\\\vdots\\F_{1'3}
\end{array}\right),
\end{equation}
respectively.
The operator $\mathscr{D}_0$ in (\ref{zy}) can be write as a $8\times3$-matrix valued differential operator:
\begin{align*}
\mathscr{D}_{0}
=\left(\begin{array}{lll} -Y_3-\mathbf{i}Y_4& -Y_1-\mathbf{i}Y_2&\ \ \ \ \ \ \ 0\\\ \ Y_1-\mathbf{i}Y_2& -Y_3+\mathbf{i}Y_4&\ \ \ \ \ \ \  0\\ -Y_7-\mathbf{i}Y_8& -Y_5-\mathbf{i}Y_6&\ \ \ \ \ \ \ 0\\\ \ Y_5-\mathbf{i}Y_6& -Y_7+\mathbf{i}Y_8&\ \ \ \ \ \ \  0\\\ \ \ \ \ \ \  0&-Y_3-\mathbf{i}Y_4& -Y_1-\mathbf{i}Y_2\\\ \ \ \ \ \ \  0&\ \  Y_1-\mathbf{i}Y_2& -Y_3+\mathbf{i}Y_4\\\ \ \ \ \ \ \  0&-Y_7-\mathbf{i}Y_8& -Y_5-\mathbf{i}Y_6\\\ \ \ \ \ \ \  0&\ \  Y_5-\mathbf{i}Y_6& -Y_7+\mathbf{i}Y_8\end{array}\right).
\end{align*}
Similarly the operator $\mathscr{D}_1$ in (\ref{eq:D-j-1}) can be write as a $6\times 8$-matrix valued differential operator:
{\small\begin{align*}
 \left(\begin{array}{llllllll} -Y_1+\mathbf{i}Y_2& -Y_3-\mathbf{i}Y_4&\ \ \ \ \ \ \ 0&\ \ \ \ \ \ \ 0&\ \ Y_3-\mathbf{i}Y_4&-Y_1-\mathbf{i}Y_2&\ \ \ \ \ \ \ 0&\ \ \ \ \ \ \ 0\\ \ \ Y_7+\mathbf{i}Y_8&\ \ \ \ \ \ \ 0& -Y_3-\mathbf{i}Y_4&\ \ \ \ \ \ \ 0&\ \  Y_5+\mathbf{i}Y_6&\ \ \ \ \ \ \ 0&-Y_1-\mathbf{i}Y_2&\ \ \ \ \ \ \ 0\\- Y_5+\mathbf{i}Y_6&\ \ \ \ \ \ \ 0&\ \ \ \ \ \ \ 0& -Y_3-\mathbf{i}Y_4&\ \ Y_7-\mathbf{i}Y_8&\ \ \ \ \ \ \ 0&\ \ \ \ \ \ \ 0&-Y_1-\mathbf{i}Y_2\\\ \ \ \ \ \ \ 0&\ \ Y_7+\mathbf{i}Y_8&\ \ Y_1-\mathbf{i}Y_2&\ \ \ \ \ \ \ 0&\ \ \ \ \ \ \ 0&\ \  Y_5+\mathbf{i}Y_6&-Y_3+\mathbf{i}Y_4&\ \ \ \ \ \ \ 0\\\ \ \ \ \ \ \  0&-Y_5+\mathbf{i}Y_6&\ \ \ \ \ \ \ 0&\ \ Y_1-\mathbf{i}Y_2&\ \ \ \ \ \ \ 0&\ \  Y_7-\mathbf{i}Y_8&\ \ \ \ \ \ \ 0&-Y_3+\mathbf{i}Y_4\\\ \ \ \ \ \ \  0&\ \ \ \ \ \ \ 0&-Y_5+\mathbf{i}Y_6&-Y_7-\mathbf{i}Y_8&\ \ \ \ \ \ \  0&\ \ \ \ \ \ \ 0&\ \ Y_7-\mathbf{i}Y_8& -Y_5-\mathbf{i}Y_6\end{array}\right) .
\end{align*}}
Thus we have
$\mathscr{D}_{0}^*=-\overline{\mathscr{D}_{0}}^t,\ \mathscr{D}_{1}^*=-\overline{\mathscr{D}_{1}}^t.$
Then by direct calculation we have \begin{equation}\label{box1}
\Box_1=\mathscr{D}_0\mathscr{D}_0^*+\mathscr{D}_1^*\mathscr{D}_1
=\left(\begin{matrix}A&0\\0&B\end{matrix}\right)
\end{equation}
with
\begin{equation*}\begin{aligned}
A&=\left(\begin{smallmatrix}
\Delta_b+\Delta_1-12\mathbf{i}\partial_{s_1}&L_1
+(Y_1+\mathbf{i}Y_2)(-Y_3-\mathbf{i}Y_4)
&(-Y_1-\mathbf{i}Y_2)(Y_5-\mathbf{i}Y_6)
&(-Y_1-\mathbf{i}Y_2)(Y_7+\mathbf{i}Y_8)\\-\overline{L_1}
+(Y_3-\mathbf{i}Y_4)(-Y_1+\mathbf{i}Y_2)
&\Delta_b+\Delta_2+12\mathbf{i}
\partial_{s_1}&(-Y_3+\mathbf{i}Y_4)(Y_5-\mathbf{i}Y_6)
&(-Y_3+\mathbf{i}Y_4)(Y_7+\mathbf{i}Y_8)\\
(-Y_5-\mathbf{i}Y_6)(Y_1-\mathbf{i}
Y_2)&(-Y_5-\mathbf{i}Y_6)(Y_3+\mathbf{i}Y_4)
&\Delta_b+\Delta_3-12\mathbf{i}\partial_{s_1}&{L_1}
+(Y_5+\mathbf{i}Y_6)(-Y_7-\mathbf{i}Y_8)\\(-Y_7+\mathbf{i}Y_8)
(Y_1-\mathbf{i}Y_2)&(-Y_7+\mathbf{i}Y_8)(Y_3+\mathbf{i}Y_4)
&-\overline{L_1}
+(Y_7-\mathbf{i}Y_8)(-Y_5+\mathbf{i}Y_6)
&\Delta_b+\Delta_4+12\mathbf{i}\partial_{s_1}
\end{smallmatrix}\right),\\
B&=\left(\begin{smallmatrix}
\Delta_b+\Delta_2-12\mathbf{i}\partial_{s_1}& {L_1}
+(-Y_3-\mathbf{i}Y_4)(-Y_1-\mathbf{i}Y_2)&(-Y_3-\mathbf{i}Y_4)
(-Y_7-\mathbf{i}Y_8)&(-Y_3-\mathbf{i}Y_4)(-Y_5-\mathbf{i}Y_6)\\
-\overline{L_1}+(Y_1-\mathbf{i}Y_2)(Y_3-\mathbf{i}Y_4)
&\Delta_b+\Delta_1+12\mathbf{i}\partial_{s_1}
&(Y_1-\mathbf{i}Y_2)(Y_7-\mathbf{i}Y_8)
&(Y_1-\mathbf{i}Y_2)(-Y_5-\mathbf{i}Y_6)\\
(Y_7-\mathbf{i}Y_8)(Y_3-\mathbf{i}Y_4)
&(-Y_7-\mathbf{i}Y_8)(-Y_1-\mathbf{i}Y_2)
&\Delta_b+\Delta_4-12\mathbf{i}\partial_{s_1}& {L_1}
+(-Y_7-\mathbf{i}Y_8)(-Y_5-\mathbf{i}Y_6)\\
(Y_5-\mathbf{i}Y_6)(Y_3-\mathbf{i}Y_4)
&(Y_5-\mathbf{i}Y_6)(-Y_1-\mathbf{i}Y_2)&-\overline{L_1}
+(Y_5-\mathbf{i}Y_6)(Y_7-\mathbf{i}Y_8)
&\Delta_b+\Delta_3+12\mathbf{i}\partial_{s_1}
\end{smallmatrix}\right),\end{aligned}\end{equation*}
where
$\Delta_b=-Y_1^2\cdots-Y_8^2,\Delta_1=-Y_1^2-Y_2^2,
\Delta_2=-Y_3^2-Y_4^2,\Delta_3=-Y_5^2-Y_6^2,\Delta_4=-Y_7^2-Y_8^2,
L_1=8(\partial_{s_2}+\mathbf{i}\partial_{s_3}).$ Because of the complexity of $\Box_1$ in (\ref{box1}), it is  not easy to obtain  its fundamental solution.

\end{document}